\documentclass[11pt,a4paper]{amsart}
\usepackage{tikz}
\usepackage{tikz-cd}
\usepackage{geometry}
\usetikzlibrary{decorations.markings,decorations.pathmorphing,cd}
\usetikzlibrary{arrows}
\usetikzlibrary{calc}
\usepackage{pgfplots}
\pgfplotsset{every axis/.append style={
                    axis x line=middle,    
                    axis y line=middle,    
                    axis line style={-,color=blue}, 
                    xlabel={$x$},          
                    ylabel={$y$},          
            }}

\usepackage{amsmath}
\usepackage{amsthm}
\usepackage{amsfonts}
\usepackage{amssymb}
\usepackage{amsrefs}
\usepackage{enumerate}
\usepackage{enumitem}
\usepackage{graphicx}
\usepackage{hyperref}
\usepackage{nicefrac}
\usepackage{cancel}

\newcommand\enet[1]{\renewcommand\theenumi{#1} 
\renewcommand\labelenumi{\theenumi}}

\DeclareMathOperator{\orb}{orb}
\DeclareMathOperator{\ncl}{NCl}

\DeclareMathOperator{\lcm}{lcm}
\DeclareMathOperator{\Sing}{Sing}

\DeclareMathOperator{\Pic}{Pic}

\DeclareMathOperator{\Hom}{Hom}

\DeclareMathOperator{\SF}{\Gamma_F}
\DeclareMathOperator{\SB}{\Gamma_S}
\DeclareMathOperator{\SBp}{\Gamma^\prime_S}
\DeclareMathOperator{\Alb}{Alb}
\DeclareMathOperator{\im}{Im}
\DeclareMathOperator{\depth}{depth}

\DeclareMathOperator{\cm}{\pi_{cl}}

\def\cQ{{\mathcal Q}}
\def\cB{{\mathcal B}}

\def\cM{\mathcal{M}}

\def\cV{\mathcal{V}}

\def\ZZ{\mathbb{Z}}
\def\TT{\mathbb{T}}

\def\CC{\mathbb{C}}
\def\DD{\mathbb{D}}

\def\QQ{\mathbb{Q}}
\def\FF{\mathbb{F}}
\def\QQ{\mathbb{Q}}

\def\PP{\mathbb{P}}
\def\hS{S'}

\def\mult{{\textrm{mult}}}

\newcommand\NS{\textrm{NS}}
\newcommand{\one}{\mathbf{1}}

\def\rightmap#1{\smash{\mathop{\rightarrow}\limits^{#1}}}

\newcommand{\Conv}{\mathop{\scalebox{1.5}{\raisebox{-0.2ex}{$\ast$}}}}%

\newtheorem{thm}{Theorem}[section]  
\newtheorem{main-thm}{Theorem}  
\newtheorem{prop}{Proposition}[section]%
\newtheorem{proper}{Condition}[section]%
\newtheorem{main-conj}{Conjecture}[section]%
\newtheorem{cor}{Corollary}[section]
\newtheorem{lemma}{Lemma}[section]

\theoremstyle{remark}
\newtheorem{rem}{Remark}[section]
\newtheorem{notation}{Notation}[section]
\theoremstyle{definition}
\newtheorem{dfn}{Definition}[section]
\newtheorem{exam}{Example}[section]

\makeatletter
\let\c@lemma\c@thm
\let\c@prop\c@thm
\let\c@propdef\c@thm
\let\c@proper\c@thm
\let\c@problem\c@thm
\let\c@conj\c@thm
\let\c@cor\c@thm
\let\c@rem\c@thm
\let\c@dfn\c@thm
\let\c@notation\c@thm
\let\c@exam\c@thm
\makeatother

\title[Quasi-projective varieties whose fundamental group...]
{Quasi-projective varieties whose fundamental group is a free product of cyclic groups}

\author[Jos\'e I. Cogolludo-Agust{\'i}n and Eva Elduque]{J.I.~Cogolludo-Agust{\'i}n and Eva Elduque}
\address{Departamento de Matem\'aticas, IUMA\\
Universidad de Zaragoza \\
C.~Pedro Cerbuna 12 \\
50009 Zaragoza, Spain.} 
\email{jicogo@unizar.es} 

\address{Departamento de Matem\'aticas, ICMAT\\ 
Universidad Aut\'onoma de Madrid \\
28049 Madrid, Spain.}
\email{eva.elduque@uam.es}

\makeindex

\begin{document}

\thanks{The authors are partially supported by PID2020-114750GB-C31, funded by 
MCIN/AEI/10.13039/501100011033. The first author is also partially funded by the Departamento de Ciencia, 
Universidad y Sociedad del Conocimiento of the Gobierno de Arag\'on 
(Grupo de referencia E22\_20R ``\'Algebra y Geometr\'{\i}a'').
}


\begin{abstract}
In the context of Serre's question, we study smooth complex quasi-projec\-tive varieties
whose fundamental group is a free product of cyclic groups. In particular, we focus on the
case of surfaces and prove the existence of an admissible map from such a quasi-projective
surface to a smooth complex quasi-projective curve. Associated with this result, we prove
addition-deletion lemmas which describe a natural operation correlating this family of
quasi-projective surfaces and groups. Our methods also allow us to produce examples of
curves in smooth projective surfaces whose complements have free products of cyclic groups
as fundamental groups, generalizing classical results on $C_{p,q}$ curves and torus-type
projective sextics, and describing the conditions under which this phenomenon occurs.
\end{abstract}

\maketitle

\section{Introduction}

\subsection*{Goals and motivation}
This paper is devoted to the general problem of describing the topology of smooth complex quasi-projective varieties. 
From the point of view of first homotopy groups, using Lefschetz-type theorems it is enough to focus on complements 
of curves in a smooth projective surface. In this context, we address Serre's question on the type of groups that can 
appear as fundamental groups of quasi-projective varieties. In particular, we consider the family of finite free products 
of cyclic groups. As it turns out, all of them can be realized as fundamental groups of quasi-projective surfaces,
but only some as fundamental groups of complements of plane projective curves (see~\S\ref{sec:main:P2}). Moreover, in 
Theorem~\ref{thm:main:intro} we describe the geometric structure of such quasi-projective surfaces as containing a natural 
fibration onto a projective curve where the finite order elements correspond with multiple fibers of the fibration.

Interest for this problem in the complex projective plane goes back to Zariski and his foundational paper~\cite{Zariski-problem}.
He considered the space of plane projective curves of degree six having six simple cusps and characterized whether these 
singular points are placed on a conic by the property $\pi_1(\PP^2\setminus D)=\ZZ_2*\ZZ_3$ for any such sextic 
$D\subset \PP^2$, or equivalently by having a $(2,3)$-toric decomposition (i.e. given by an equation $f^2_3+f^3_2=0$ where 
$f_i$ is a generic form of degree $i$ in $\CC[x,y,z]$). This is equivalent to the existence of a morphism 
$F:\PP^2\setminus D\to \PP^1\setminus \{[1:-1]\}$, defined by $F(x,y,z)=[f^2_3:f^3_2]$ with two 
multiple fibers. In the 70's, Oka proved his classical result on $C_{p,q}$ curves 
in~\cite[\S8]{Oka-some-plane-curves} (see also Dimca~\cite[Prop. \S4(4.16)]{Dimca-singularities}),
exhibiting a family of irreducible curves 
\begin{equation}
\label{eq:cpq}
C_{p,q}=\{(x^p+y^p)^q+(y^q+z^q)^p=0\}
\end{equation}
for any $p,q>1$ coprime such that $\pi_1(\PP^2\setminus C_{p,q})\cong\ZZ_p*\ZZ_q$.
This was generalized in the 80's by N\'emethi in~\cite{Nemethi-fundamental}.

This problem appeared in several other contexts, mainly focusing on a possible connection between 
Alexander polynomials of complements $\PP^2\setminus D$ and toric (or more generally quasi-toric)
decompositions 
(\cite{Kulikov-albanese,ji-Libgober-mw,kawashima-kenta-torus,Oka-geometry-02,Oka-Pho-classification,Oka-Pho-fundamental,Tokunaga-albanese,Tokunaga-torus-albanese,ACM-multiple-fibers}). 
For instance, Oka's conjecture~\cite{Oka-Eyral-fundamental} 
(discussed and proved by Degtyarev in~\cite{Degtyarev-OkaConjectureII,Degtyarev-OkaConjecture}), 
states that an irreducible plane sextic is of torus type if and only if its Alexander polynomial is 
non-trivial.

A classical tool to describe the topology of smooth quasi-projective surfaces is the existence of morphisms onto smooth 
curves (a complex quasi-projective manifold of dimension 1). This is described in the Castelnuovo-de Franchis theorem for
the existence of morphisms onto smooth curves of genus $g\geq 2$, and in Arapura's structure theorem~\cite{Arapura-geometry}, 
as well as its twisted version~\cite{ACM-characteristic}. If $X$ denotes a smooth projective surface 
and $D\subset X$ a (non-empty) reduced curve, the latter paper uses certain properties of the fundamental group to 
describe the existence of a dominant morphism $F$ from $X\setminus D$ to a smooth projective curve $S$, 
which induces an orbifold structure $S^{\orb}$ on $S$ determined by the
multiplicity of the fibers of $F$. Since the fibers of this morphism are generically connected, the authors show that 
the fundamental group $\pi_1(X\setminus D)$ surjects onto the orbifold fundamental group of $S^{\orb}$, 
which, since $S\setminus F(X\setminus D)\neq \emptyset$, is a finitely generated free product of cyclic groups. 
The extremal case for this morphism induced by $F$ occurs  when the surjection becomes an 
isomorphism between $\pi_1(X\setminus D)$ and the orbifold fundamental group of $S^{\orb}$. 
This motivates our use of these techniques for the study of quasi-projective varieties whose fundamental 
group is a free product of cyclic groups, and raises the question of the existence of such a morphism 
$F$ realizing an isomorphism between (orbifold) fundamental groups in this setting. Indeed, similar 
connections between fundamental groups of smooth varieties and orbifold fundamental groups have been 
studied by Arapura~\cite{Arapura-toward} and Catanese~\cite{Catanese-differentiable} in the projective 
case ($D=\emptyset$), and by Bauer~\cite{Bauer-irrational} and Catanese~\cite{Catanese-Fibred} in the 
quasi-projective case where $\pi_1(X\setminus D)$ is free (or more generally, it admits an epimorphism 
onto a free group with a finitely generated kernel).

In this paper, we provide both necessary and sufficient geometric conditions for a smooth quasi-projective variety to 
have a free (finite) product of cyclic groups as its fundamental group. However, before stating our results in full 
generality, let us provide some intuition by exhibiting some of their consequences in the case of plane curve complements. 
We show that if the fundamental group of the complement of a curve in $\PP^2$ is a free product of cyclic groups, then 
it is isomorphic to $\FF_r*\ZZ_p*\ZZ_q$ for some $r\geq 0$ and some $p,q\geq 1$ such that $\gcd(p,q)=1$. 
We prove the following structure theorem, which provides necessary conditions and greatly constrains the type of 
polynomial equations that can give rise to curves in $\PP^2$ whose fundamental group of their complement is a free 
product of two non-trivial cyclic groups. The following is an immediate consequence of Corollary~\ref{thm:main:P2}.

\begin{cor}
\label{cor:pq}
If a curve complement in $\PP^2$ has fundamental group $\ZZ_p*\ZZ_q$, with $p,q\in\ZZ_{>1}$ coprime integers, 
then the curve is given by a polynomial equation of the form $f_p^q+f_q^p=0$, for some $f_p$ and $f_q$ homogeneous 
polynomials in $\CC[x,y,z]$ of degrees $p$ and $q$ with no common factors, and such that  neither of them is a $k$-th 
power of another polynomial for any $k\geq 2$. 
\end{cor}

A partial converse is also proved by showing that the fundamental group of the complement 
of any curve defined by $f_p^q+f_q^p=0$ is \emph{generically} a free product of cyclic groups.

\begin{thm}
\label{thm:Okapq}
Let $f_p$ (resp. $f_q$) be a homogeneous polynomial of degree $p$ (resp. $q$) in $\CC[x,y,z]$ with $\gcd(p,q)=1$. Assume that
\begin{itemize}
\item $f_p$ and $f_q$ define an admissible map $F=[f_p^q:f_q^p]:\PP^2\setminus \mathcal B\to\PP^1$, where $\mathcal B=V(f_p)\cap V(f_q)$ is a finite set.
\item The multiple fibers of $F$ lie over a subset of $\{[0:1],[1:0]\}$ (this always holds if $p,q\geq 2$).
\end{itemize}
Let $r\geq 0$, and let $C_0,\ldots, C_r$ be the closures in $\PP^2$ of $r+1$ distinct
generic fibers of $F$. Let $C=\cup_{i=0}^r C_i$.
Then,
\begin{equation*}
\pi_1(\PP^2\setminus C)=\FF_r*\ZZ_p*\ZZ_q.
\end{equation*}
Moreover, assume that $V(f_p)$ is irreducible and $\pi_1(\PP^2\setminus V(f_p))\cong \ZZ_p$, then
\begin{equation*}
\pi_1\left(\PP^2\setminus (C\cup V(f_p))\right)=\FF_{r+1}*\ZZ_p.
\end{equation*}
\end{thm}

In particular, this provides infinitely many examples of curves in $\PP^2$ like the ones in the previously mentioned 
examples by Zariski, Oka and Némethi, that is, such that the fundamental group of their complement is a free product of 
two non-trivial cyclic groups. Zariski's and Oka's results were proved using braid monodromy computations (which can be 
very complicated) applied to very expertly chosen specific examples. The methods presented here are different and their 
scope is wider, as they rely on exploiting the extra structure on these varieties endowed by morphisms to smooth complex 
quasi-projective curves (such as $F$ in Theorem~\ref{thm:Okapq}). These morphisms are constructed using properties of 
the fundamental group itself, not invariants of it such as the Alexander polynomial. 

\subsection*{Overview of the main results}

The main objects of this paper are curves in a smooth projective surface.
Whenever we refer to a divisor as a curve, the divisor is meant to have a reduced structure.
One of the main theorems of this paper provides geometric necessary conditions for the complement of a 
curve $D$ in a smooth projective surface $X$ to have a free product of cyclic groups as its fundamental group. 
These conditions include the existence of an \emph{admissible map} to a smooth curve $S$ 
(see Section~\ref{sec:admissible} for the definition of admissible map), and are stated using orbifold 
notation $S_{(n+1,\bar m)}$, where $\Sigma_0\subset S$ is the set of $\#\Sigma_0=n+1$ points removed 
and $\bar m=(m_1,\dots,m_s)$ represents the orbifold structure on $s$ points of $S\setminus \Sigma_0$
(see Section~\ref{sec:orbifold} for the relevant definitions). 

\begin{thm}\label{thm:main:intro}
Let $X$ be a smooth connected projective surface, and let $D\subset X$ be a curve.
Suppose that $\pi_1(X\setminus D)\cong\FF_{r}*\ZZ_{m_1}*...*\ZZ_{m_s}$ is infinite.
Then, there exists a smooth projective curve $S$ of genus $g_S$ and an admissible map $F:X\dashrightarrow S$ such that:
\begin{enumerate}
\enet{\textrm(\roman{enumi})}
\item \label{lemma:existencePencil-2}
$F$ induces an orbifold morphism
\begin{equation*}
F_|:X\setminus D\to S_{(n+1,\bar m)},
\end{equation*}
where $S_{(n+1,\bar m)}$ is maximal with respect to $F_|$, $n\geq 0$ and $\bar m=(m_1,\ldots, m_s)$.
\item \label{lemma:existencePencil-3}
$F_*:\pi_1(X\setminus D)\to \pi_1^{\orb}(S_{(n+1,\bar m)})$ is an isomorphism.
\item \label{lemma:existencePencil-1}
$D=D_f\cup D_t$, where
\begin{itemize}
\item $D_f=\overline{F^{-1}(\Sigma_0)}$ is a fiber-type curve which is the union of the $n+1$ fibers
above the distinguished points $\Sigma_0\subset S$, with $n=r-2g_S$ and
\item the meridians of $D_t$ are trivial in~$\pi_1(X\setminus D)$.
\end{itemize}
\end{enumerate}
\end{thm}

In particular, the maximality condition of part~\ref{lemma:existencePencil-2} implies that 
$F_|:X\setminus D\to S\setminus\Sigma_0$ is a surjective algebraic map with exactly $s$ multiple 
fibers of multiplicities $(m_1,\ldots,m_s)$, determined by the torsion of $\pi_1(X\setminus D)$.
Theorem~\ref{thm:main:intro} is proved in Section~\ref{sec:main}.

The existence of the admissible map $F$ is guaranteed by~\cite{Arapura-geometry,ACM-characteristic}
if $r\geq 1$ and $\pi_1(X\setminus D)\neq \ZZ$. The case $\pi_1(X\setminus D)=\ZZ$ has to be
considered separately. If $r=0$ the structure of the group $\pi_1(X\setminus D)$
is used to construct a finite covering with free fundamental group. We then show that the Albanese morphism of this
covering produces an admissible map which, by the functoriality of the Albanese morphism, descends to the desired
admissible map~$F$ (after normalization in the target curve and Stein factorization).
Moreover, Corollary~\ref{cor:non-fiber-to-fiber} discusses the role of the divisor~$D_t$.

Theorem~\ref{thm:main:intro} is extended in Section~\ref{sec:main:extensions} to the case where 
$\pi_1(X\setminus D)$ is finite as long as $X$ is simply connected, under some extra assumptions that are 
always satisfied if $X=\PP^2$. As a result, we prove a refinement of 
Theorem~\ref{thm:main:intro} for the case $X=\PP^2$ in Section~\ref{sec:main:P2} (Corollary~\ref{thm:main:P2}) and
show that in this case $\pi_1(X\setminus D)\cong \FF_r*\ZZ_p*\ZZ_q$ for some $p,q\in\ZZ_{\geq 1}$ coprime and
$r\in\ZZ_{\geq 0}$, and that $D=D_f$ is a union of irreducible fibers of a pencil $F:\PP^2\dashrightarrow\PP^1$. The already stated Corollary~\ref{cor:pq} is a particular case ($r=0$ and $p,q>1$) of Corollary~\ref{thm:main:P2}.

The second type of main results are referred to as Addition-Deletion Theorems in Section~\ref{sec:additiondeletion}.
Consider $U$ a smooth quasi-projective surface and $F:U\rightarrow S$ an admissible map to a smooth projective curve $S$,
and define $U_B=U\setminus F^{-1}(B)$ for any finite subset $B\subset S$. In this context, we prove a 
Deletion Lemma~\ref{lemma:removing-fiber-orbi} that describes the fundamental group of $U_B$ if $F_|: U_{B\cup\{P\}}\to S$ induces an isomorphism between (orbifold) fundamental groups. 
 This is done regardless of whether $P\in B_F$ or not, for $B_F$ the set of atypical 
values of $F$, that is, whether $F^{-1}(P)$ is a typical fiber or not.
We also prove the following Generic Addition-Deletion Lemma in Section~\ref{sec:additiondeletion}

\begin{thm}[Generic Addition-Deletion Lemma]
\label{lemma:generic-fiber-orbi}
Let $U$ be a smooth quasi-projective surface and let $F:U\rightarrow S$ be an admissible map to a smooth
projective curve $S$.
Assume $B\subset S$, where $\# B=n\geq 1$, and let $P\in S\setminus \left(B_F\cup B\right)$.
Consider $S_{(n+1,\bar m)}$ (resp. $S_{(n,\bar m)}$) the maximal orbifold structure of $S$ with respect to 
$F_|:U_{B\cup\{P\}}\to S\setminus (B\cup \{P\})$ (resp. $F_|:U_{B}\to S\setminus B$).

Then the following are equivalent:
\begin{itemize}
 \item
$F_*:\pi_1(U_B)\to\pi_1^{\orb}(S_{(n,\bar m)})$ is an isomorphism,
 \item
$F_*:\pi_1(U_{B\cup\{P\}})\to\pi_1^{\orb}(S_{(n+1,\bar m)})$ is an isomorphism.
\end{itemize}
Moreover, in that case,
\begin{equation*}
\pi_1(U_{B\cup\{P\}})\cong \ZZ*\pi_1(U_B).
\end{equation*}
\end{thm}
As a consequence of this, we prove that the fundamental group of the complement of $r$ generic fibers
of a primitive polynomial map from $\CC^2$ onto $\CC$ is free of order~$r$ in Corollary~\ref{cor:polynomialmap}.

We finally devote Section~\ref{sec:applications} to a number of applications of these
results to the calculation of fundamental groups of complements of projective curves,
noting how the results in Section~\ref{sec:additiondeletion}
provide sufficient geometric conditions for these groups to be a free product of cyclic groups.
In particular, in Section~\ref{sec:Cpq} we prove the already stated Theorem~\ref{thm:Okapq}.
The key ingredients in the proof of this theorem are the Addition-Deletion 
results, which allow us to avoid any Zariski-Van Kampen/braid monodromy calculations.

Theorem~\ref{thm:Okapq} brings together several examples known in the literature. For instance, we apply it 
in Section~\ref{sec:conics} to provide a Zariski-Van Kampen-free proof of a result on the fundamental group of a 
union of lines and conics due to Amram-Teicher~\cite[Thm. 2.2, 2.5]{Amram-ErratumFundamentalQuadric}.
As a last application, in~\ref{sec:TorusType} we generalize a classical result due to Oka-Pho in~\cite{Oka-Pho-fundamental} 
on fundamental groups of maximal tame torus-type sextics.

Let us remark that, after the initial preparation of this paper, more progress has been made in this topic,
more concretely, about potential refinements of the sufficient conditions from Theorem~\ref{thm:Okapq} for the
fundamental group of the complement in $\PP^2$ of a plane curve to be a free product of cyclic groups.
More concretely, the examples of \cite{ji-Eva-Zariski} show that those sufficient conditions would need to
involve more than local invariants.

\begin{rem}
The results in this paper could be generalized to the quasi-K\"ahler (complements of normal crossing divisors in a 
compact K\"ahler surface) case in the following sense. Section~\ref{sec:preliminaries} can be extended to 
quasi-K\"ahler surfaces. These results would provide proofs of analogues to Theorems~\ref{thm:main:intro} 
and~\ref{lemma:generic-fiber-orbi} in this context. However, our main interest in this paper is the pair $(X,D)$,
where $X$ is a projective surface and $D$ is a curve on $X$ (not necessarily with normal crossings). 
The statements of Theorems~\ref{thm:main:intro} and~\ref{thm:Okapq} and Corollary~\ref{cor:pq} describe 
some properties of the curve $D$ in~$X$.
\end{rem}

\subsection*{Acknowledgments}
The authors would like to thank Enrique Artal, Mois\'es Herrad\'on Cueto, Anatoly Libgober,
and Jakub Witaszek for useful conversations. We also thank an anonymous referee for
useful comments that have helped improve the presentation.

\section{Preliminaries}\label{sec:preliminaries}
A short exposition on the main concepts and tools of this paper will be given in this section in order to fix 
notation and for the sake of completeness.

\subsection{Admissible maps on $X$}\label{sec:admissible}
Following Arapura~\cite{Arapura-geometry}, we call a surjective morphism $F:U\to \hS$ from a smooth 
quasi-projective surface $U$ to a smooth quasi-projective curve $\hS$ \textit{admissible} if it admits a 
surjective holomorphic enlargement $\hat F:\hat U\to S$ with connected fibers, 
where $\hat U$ and $S$ are smooth compactifications of $U$ and $\hS$ respectively. 
As a consequence, note that the admissible map $F$ has connected generic fibers,
and in fact, both conditions are equivalent (cf. \cite[Remark 2.2]{ji-Eva-Zariski}).

If $U$ is a Zariski dense subset of a smooth projective surface $X$, then $F$ defines a rational map 
$F:X\dashrightarrow S$ which can be extended to an admissible map
$F:X\setminus \mathcal B\to S$ on a maximal open
set $X\setminus \mathcal B$, where $\mathcal{B}$ is the finite subset
of base points. Note that $F$ must be surjective
and that $\mathcal B=\emptyset$ unless $S=\PP^1$. For convenience, this rational map will
also be referred to as \textit{admissible}.

\begin{rem}
\label{rem:h1-torsion}
An admissible map $F:X\setminus\mathcal B\rightarrow S$ has connected generic
fibers and hence it induces an epimorphism
$F_*:\pi_1(X\setminus\mathcal B)=\pi_1(X)\to \pi_1(S)$.
In particular, $H_1(X;\CC)=0$ implies $S=\PP^1$.
\end{rem}

Given $P\in S'\subset S$, the fiber $F^{-1}(P)\subset U$ defines an algebraic curve $C_P$. 
Assume $B\subset S'$ is a finite set. We will denote $C_B=\cup_{P\in B}C_P$. Any such curve will be referred to as a 
\emph{fiber-type curve}. It is well known that the minimal set of values $B_F$ for which 
$F:U\setminus C_{B_F}\to \hS\setminus B_F$ is a locally trivial fibration is finite~\cite{Thom-ensembles}. 
The points in $B_F$ are called \emph{atypical} values of $F:U\to \hS$.
We will distinguish between $F^{*}(P)$ as the pulled-back divisor and $C_P$ as 
its reduced structure. Using this notation, one can describe the set of multiple fibers as
\begin{equation}
\label{eq:multiple}
M_F=\{P\in \hS\mid F^{*}(P)=mD, m>1, \textrm{ for some effective divisor } D\}\subset B_F.
\end{equation}
Note that in general, the effective divisor $D$ in~\eqref{eq:multiple} need not be reduced. If $P\in S$ the 
\emph{multiplicity} 
of $F^{*}(P)$ is defined as $m\geq 1$ if $F^{*}(P)=mD$ for some $D$ and whenever $F^{*}(P)=m'D'$, then~$m'\leq m$.

\begin{rem}
\label{rem:2fibers}
If $X$ is a simply connected surface and $F:X\dashrightarrow S$ is an admissible map, then $S=\PP^1$ by 
Remark~\ref{rem:h1-torsion} and an analogous argument to the one given in the proof 
of~\cite[Prop. 2.8]{ji-Libgober-mw} shows that the number of multiple fibers of $F$ cannot exceed two.
\end{rem}

From now on, we will use the following notation.

\begin{notation}\label{not:XB}
Let $F:U\to \hS$ be an admissible map from a smooth quasi-projective surface $U$ to a smooth quasi-projective curve $\hS$, 
and let $B\subset \hS$ be a finite set. We denote by $U_B:=U\setminus C_B$. Analogously, if $F:X\dashrightarrow S$ is an admissible rational map from a smooth projective surface $X$ to a smooth projective curve $S$ and $B\neq\emptyset$, one defines $X_B$ as 
$U_B$ for $U=X\setminus \mathcal B$. Note that $X_B=X\setminus\left(\cup_{P\in B}\overline{F^{-1}(P)}\right)$.
\end{notation}

\subsection{Fundamental groups of quasi-projective varieties.}\label{sec:fundamentalgroups}
Let $X$ be a smooth quasi-projective surface and let $D=\cup_{i\in I} D_i$ be a curve in $X$, where $D_i$ are its
irreducible components. 

When studying $\pi_1(X\setminus D,p)$ 
one has the following generating homotopy classes of loops: Take a regular point $p_i$ on $D_i$ and consider a disk $\DD_i\subset X$ transversal to $D_i$ at $p_i$
and such that $\DD_i\cap D=\{p_i\}$. Let $\hat p_i\in \partial \DD_i$ and consider $\hat \gamma_i$ a loop 
based at $\hat p_i$ around $\partial \DD_i$ travelled in the positive orientation. 
Define $\delta_i:[0,1]\to X\setminus D$ a path in $X\setminus D$ starting at the base point $\delta_i(0)=p$ and ending 
at $\delta_i(1)=\hat p_i$. Denote by $\bar\delta_i$ the reversed path defined as usual as
$\bar\delta_i(t):=\delta_i(1-t)$, $t\in [0,1]$ starting at $p_i$ and ending at $p$.
The following loop $\gamma_i:=\delta_i \star \hat\gamma_i \star \bar \delta_i$ is based at $p$ and 
defines a homotopy class called a \emph{meridian} around $D_i$. The following two results are well known.

\begin{lemma}
\label{lemma:meridians-conjugated}
Let $\gamma$ be a meridian around $D_i$. A homotopy class $\gamma'$ is a meridian around $D_i$ if and 
only if $\gamma'$ is in the conjugacy class of $\gamma$ in $\pi_1(X\setminus D,p)$.
\end{lemma}

\begin{proof}
See \cite{Nori-Zariski}, also~\cite[Prop. 1.34]{ji-pau} for a proof.
\end{proof}

\begin{lemma}
\label{lemma:meridians-generators}
Consider $X_i:=X\setminus (\cup_{j\in I\setminus\{i\}}D_j)$ and
$(j_i)_*:\pi_1(X\setminus D,p)\to \pi_1(X_i,p)$ induced by the
inclusion $X\setminus D\hookrightarrow X_i$. Then $(j_i)_*$ is surjective,
and $\ker (j_i)_*$ is the normal closure of
any meridian $\gamma_i$ in~$\pi_1(X\setminus D,p)$.
In particular, if $X$ is simply connected, then any set of the form $\{\gamma_i\}_{i\in I}$ normally generates~$\pi_1(X\setminus D,p)$, where $\gamma_i$ is a meridian around $D_i$ for all $i\in I$.
\end{lemma}

\begin{proof}
See~\cite{Nori-Zariski}. Also, as a consequence of~\cite[Lemma 2.3]{Shimada-Zariski}.
\end{proof}

\subsection{Homology of the complement}
Consider $X$ a smooth projective surface and let $D=D_0\cup\dots\cup D_r\subset X$ be the decomposition of 
a curve $D$ into its irreducible components.
Using excision and Lefschetz duality, the homology exact sequence of $(X,X\setminus D)$ becomes 
$$
H_2(X;\ZZ) \to H^2(D;\ZZ)\to H_1(X\setminus D;\ZZ)\to H_1(X;\ZZ)\to 0.
$$
where $H^2(D;\ZZ)\cong\ZZ^{r+1}$ is generated by the cohomology classes of each $D_i$
irreducible component of~$D$. Hence, if $H_1(X;\ZZ)=0$ (resp. if $H_1(X;\QQ)=0$), one obtains
\begin{equation}
\label{eq:j}
H_2(X;\ZZ)\ \rightmap{j}\ \ZZ^{r+1} \to H_1(X\setminus D;\ZZ)\to 0
\end{equation}
(resp. with $\QQ$-coefficients), where $j(C)=\sum_{i=0}^r (C,D_i)_{X}D_i$ (see for instance~\cite{Bredon-topology}). In particular,
\begin{equation}
\label{eq:H1}
H_1(X\setminus D;\ZZ)=\ZZ^{r+1}/\im j \text{ \ \ \ (resp. }H_1(X\setminus D;\QQ)=\QQ^{r+1}/\im j\text{)}.
\end{equation}

The following condition on the irreducible components of a curve allows for a
particularly simple description of the first homology of~$X\setminus D$.
In this paper, whenever we refer to a divisor as a curve, the divisor is meant
to have a reduced structure.

\begin{proper}
\label{eq:numequiv}
The curve $D$ decomposes into irreducible components as $D=\cup_{i=0}^r D_i$,
and the irreducible components are such that,
$$m_iD_i\equiv m_jD_j \quad \textrm{ for some} \quad m_0,\dots,m_r\in \ZZ_{>0},$$
where $\equiv$ here means \emph{numerical equivalence}.
\end{proper}

\begin{exam}
\label{exam:condition}
The following are typical sources of examples satisfying Condition~\ref{eq:numequiv}:
(i) For any $D$, if the surface $X$ is such that $\NS(X)=\Pic(X)/\Pic^0(X)=\ZZ$.
(ii) For any $X$, if there exists an admissible map $F$ from $X$ onto a curve and $m_i\geq 1$ such that
$D=\cup_{i=0}^r D_i$ and $m_iD_i$ is a (multiple if $m_i>1$) fiber of~$F$.
(iii) For any $X$, whenever $D$ is irreducible, since $\im(j)\subset \ZZ D$.
\end{exam}

The following result is immediate from~\eqref{eq:H1}.

\begin{lemma}
\label{lem:cyclichomology}
If $X$ is a smooth projective surface such that $H_1(X;\ZZ)$ is finite, and $D$ is a curve
satisfying Condition~\ref{eq:numequiv}, then $H_1(X\setminus D;\QQ)\cong \QQ^r$.
Moreover, if $H_1(X;\ZZ)=0$, then $H_1(X\setminus D;\ZZ)\cong \ZZ^r\times \ZZ_d$,
where $d\in\ZZ_{>0}$ is determined by the components~$D_i$.
\end{lemma}

\begin{rem}\label{rem:coprime}
In particular, if $D\subset X=\PP^2$ is a curve with $r+1$ irreducible components, then
$H_1(X\setminus D;\ZZ)\cong \ZZ^r\times \ZZ_d$,
where $d$ is the greatest common divisor of the degrees of the irreducible  components of
$D$ (see~\cite[\S4 Prop. (1.3)]{Dimca-singularities}).
\end{rem}

\subsection{Orbifold fundamental groups and orbifold admissible maps}
\label{sec:orbifold}
In this section we will define the concept of orbifold fundamental groups and
orbifold morphisms induced by admissible maps.
As a word of caution, the word orbifold might be misleading, since we do not need
to develop any theory of orbifolds
or $V$-manifolds in this context. This will become clear throughout the section.
The first concept is a group theoretical
object associated with a smooth projective curve (or in more generality with a
projective manifold) and a divisor on it.
The second concept is purely geometric and only reflects the existence of non-reduced
fibers of an admissible map.

Consider a smooth projective curve $S$ of genus $g$ and choose a labeling map 
$\varphi:S\to\ZZ_{\geq 0}$ such that $\varphi(P)\neq 1$ only for a finite number of points, say 
$\Sigma=\Sigma_0\cup\Sigma_+\subset S$ for which $\varphi(P)= 0$ if $P\in \Sigma_0$ and 
$\varphi(Q)=m_{Q}>1$ if $Q\in \Sigma_+$. In this context, we will refer to this as an \emph{orbifold structure on} $S$.
This structure will be denoted by $S_{(n+1,\bar m)}$, where $n+1=\#\Sigma_0$, and $\bar{m}$ is a $(\#\Sigma_+)$-tuple 
whose entries are the corresponding $m_Q$'s. The \emph{orbifold fundamental group} associated with $S_{(n+1,\bar m)}$,
denoted by $\pi_1^{\orb}(S_{(n+1,\bar m)})$, is the quotient of
$
\pi_1(S\setminus \Sigma)$ by the normal closure of the subgroup $
\langle \mu_P^{\varphi(P)}, P\in \Sigma\rangle,
$
where $\mu_P$ is a meridian in $S\setminus \Sigma$ around $P\in \Sigma$.
Note that $\pi_1^{\orb}(S_{(n+1,\bar m)})$ is hence generated by
$
\{a_i,b_i\}_{i=1,\dots,g} \cup \{\mu_P\}_{P\in \Sigma}
$
and presented by the relations
\begin{equation}
\label{eq:rels}
\mu_P^{m_P}=1, \quad \textrm{ for } P\in \Sigma_+, \quad \textrm{ and } \quad
\prod_{P\in \Sigma}\mu_P=\prod_{i=1,\dots,g}[a_i,b_i]
\end{equation}
for appropriately chosen $
\{a_i,b_i\}_{i=1,\dots,g}$ and meridians $\{\mu_P\}_{P\in \Sigma}
$. In the particular case where $\Sigma_0\neq \emptyset$, \eqref{eq:rels} shows that
$\pi_1^{\textrm{orb}}(S_{(n+1,\bar m)})$ is a free product of cyclic groups as follows
\begin{equation*}
\pi_1^{\textrm{orb}}(S_{(n+1,\bar m)})\cong \pi_1(S\setminus\Sigma_0)\Conv\left(
\Conv_{P\in\Sigma_+}\left(\frac{\ZZ}{m_P\ZZ}\right)\right)\cong
\FF_r*\ZZ_{m_1}*\dots*\ZZ_{m_s},
\end{equation*}
where $r=2g-1+\# \Sigma_0=2g+n$,

\begin{dfn}
Let $S$ be a smooth projective curve endowed with an orbifold structure.
We refer to the orbifold fundamental group of $S$ as an \emph{open orbifold group}
of $S$ when the orbifold structure on $S$ is such
that~$\Sigma_0\neq\emptyset$, or equivalently, $n\geq 0$.
\end{dfn}

\begin{dfn}
The \emph{orbifold Euler characteristic} of $S_{(n+1,\bar m)}$ is given as
$$
\chi^{\orb}(S_{(n+1,\bar m)}):=2-2g-(n+1)-\sum_i\left( 1-\frac{1}{m_i}\right)
=1-(s+2g+n)+\sum_i\frac{1}{m_i}.
$$
\end{dfn}

\begin{dfn}
Let $X$ be a smooth algebraic variety. A dominant algebraic morphism $F:X\to S_{(n+1,\bar m)}$ is called an
\emph{orbifold morphism} if for all $P\in S$ such that $\varphi(P)>0$, the divisor $F^*(P)$ is a
$\varphi(P)$-multiple. The orbifold $S_{(n+1,\bar m)}$ is said to be \emph{maximal} (with respect to $F$)
if $F(X)=S\setminus \Sigma_0$ and for all $P\in F(X)$ the divisor $F^*(P)$ is not an $n$-multiple for any $n > \varphi(P)$.
\end{dfn}

The following result is well known (see for instance~\cite[Prop. 1.4]{ACM-multiple-fibers})
\begin{rem}\label{rem:inducedorb}
Let $F:X\to S_{(n+1,\bar m)}$ be an orbifold morphism. Then, $F$ induces a morphism
$
F_*:\pi_1(X)\to\pi_1^{\orb}(S_{(n+1,\bar m)}).
$
Moreover, if the generic fiber of $F$ is connected, then $F_*$ is surjective.
\end{rem}

The following lemma extends~\cite{Nielsen-commutator}. 

\begin{lemma}
\label{lemma:kernel}
Consider $G=\ZZ_{m_1}*\dots *\ZZ_{m_s}$, with $s\geq 2$, $m_i>1$, and $m:=\lcm(m_i,i\in I)$, for $I=\{1,\dots,s\}$. 
Let $\cm:G\to \ZZ_m$ be the natural epimorphism of $G$ onto its maximal cyclic quotient. 
Then $\ker(\cm)\cong \FF_\rho$, a free group of rank 
$$\rho=1-m+m\sum_{i\in I}\left(1-\frac{1}{m_i}\right)=1-m\chi^{\orb}(\PP^1_{(1,\bar m)}).$$ 
\end{lemma}

\begin{proof}
It is straightforward using Reidemeister-Schreier techniques and induction over
the number of distinct prime factors of $m$.
\end{proof}

The following well-known result is a generalization of~\cite[Lemma 4]{Oka-Pho-fundamental}.

\begin{lemma}
\label{lemma:EpiImpliesIso}
Let $G$ be a finitely generated free product of cyclic groups. Then, $G$ is a
Hopfian group, i.e. every endomorphism of $G$ which is an epimorphism is an isomorphism.
\end{lemma}

\begin{proof}
Consider $G=\FF_r*\ZZ_{m_1}*\dots*\ZZ_{m_s}$, where $m_i\geq 2$ for any $i\in I=\{1,\dots,s\}$.
Let $J=\FF_r$ and $H=\ZZ_{m_1}*\dots*\ZZ_{m_s}$. Finitely generated free groups are Hopfian,
so $J$ is Hopfian. $H$ is a free product of finitely many finite groups, so it is virtually
free and thus residually finite. Finitely generated residually finite groups are Hopfian,
so $H$ is Hopfian. By \cite{DeyNeumann}, the free product of two finitely generated Hopfian
groups is Hopfian.
\end{proof}

\subsection{Fundamental groups of complements of fiber-type curves}
\label{sec:pi1-fiber-type}
In this section, $F:U\to S$ is going to be an admissible map from a smooth quasi-projective
surface to a smooth projective curve $S$ of genus $g_S$. Following~\cite{Kashiwara-fonctions}
we say the generic fiber $F^{-1}(P)$ of an admissible map is of \emph{type $(g_F,s_F)$}
if $F^{-1}(P)$ is homeomorphic to a smooth projective curve of genus $g_F$ with $s_F$ points
removed. We will denote by
\begin{equation}
\label{eq:pi-riemann}
\mathop{\Omega}\nolimits_{(g,s)}=\langle a_1,\dots,a_{g},b_1,\dots,b_{g},x_0,\dots,x_{s-1}:
\Pi_{i=1}^{g}[a_i,b_i]=\Pi_{j=0}^{s-1}x_j\rangle
\end{equation}
the fundamental group of a smooth projective curve of genus $g$, 
where $x_0, \dots, x_{s-1}$ are meridians around its $s$ punctures.

As above, consider the admissible map $F:U\to S$, $B=\{P_0,P_1,\dots,P_n\}\subset S$, $n\geq 0$. Let
$\SB=\{\gamma_1,\dots,\gamma_r\}$, $r=2g_S+n$ be a set of loops in $\pi_1(U_B)$ such that:
\begin{enumerate}
 \item\label{dfn:adapted-1} 
 $\{F_*(\gamma_k)\}_{k=1}^{r}$ generates $\pi_1(S\setminus B)\cong \FF_{r}$ for all $k=1,\ldots, r$,
 \item\label{dfn:adapted-2}
 the loops $\gamma_k\in\pi_1(U_B)$ result from lifting a meridian around $P_k\in S$, for $k=1,\ldots,n$,
 \item\label{dfn:adapted-3}
 the product $\tilde\gamma=\prod_{j=1}^{g_S}[\gamma_{n+2j-1},\gamma_{n+2j}]\left( \prod_{i=1}^n\gamma_i \right)^{-1}$ is 
 such that $F_*(\tilde\gamma)$ is a meridian around~$P_0$.
 \item\label{dfn:adapted-4} For every $P_k\in B$ such that $F^*(P_k)$ is not a multiple fiber, $\gamma_k$ is a product of 
meridians (positively or negatively oriented) about irreducible components of $C_{P_k}\subset C_B$ for all $k=1,\ldots,n$. 
In the particular case where $C_{P_k}$ is irreducible, $\gamma_k$ is a positively oriented meridian about $C_{P_k}$.
\end{enumerate}
On the other hand, if $P\in S\setminus (B_F\cup B)$, then the typical fiber $F^{-1}(P)$ is a smooth
curve of type $(g_F,s_F)$ and $\pi_1(F^{-1}(P))\cong \Omega_{(g_F,s_F)}$ as 
in~\eqref{eq:pi-riemann}. Let $\SF$ denote the image of such a set of generators as in~\eqref{eq:pi-riemann} by the 
homomorphism induced by the inclusion $\iota:F^{-1}(P)\hookrightarrow U_B$.

\begin{dfn}\label{dfn:adapted}
Any set of loops $\SF$ (resp. $\SB$) obtained as in the construction above will be referred to as an
\emph{adapted geometric set of fiber (resp. base) loops} w.r.t. $F$ and $B$.
\end{dfn}

The following shows that adapted geometric sets of fiber (resp. base) loops exist for admissible maps.

\begin{lemma}\label{lemma:admissible}
Let $F:U\to S$ be an admissible map from a smooth quasi-projective surface $U$ to a smooth projective curve $S$ of genus $g_S$.
Consider $B=\{P_0,P_1,\dots,P_n\}\subset S$, $n\geq 0$. Then, there exists $\SB=\{\gamma_1,\dots,\gamma_r\}$, $r=2g_S+n$ (resp. $\SF$) 
an adapted geometric set of base (resp. fiber) loops w.r.t. $F$ and $B$.
\end{lemma}

\begin{proof}
The statement with respect to $\SF$ follows by construction. As for $\SB$, let us choose a set of loops in 
$\pi_1(S\setminus B)=\mathop{\Omega}\nolimits_{(g_S,n+1)}$ satisfying~\eqref{eq:pi-riemann}.
Since $F$ has connected fibers and is algebraic, $F_*:\pi_1(U_B)\to \pi_1(S\setminus B)\cong\FF_r$ is surjective and one can 
choose liftings $\gamma_k$ satisfying properties~\eqref{dfn:adapted-1}--\eqref{dfn:adapted-3} above. 

To see that we may choose $\SB$ so that it also satisfies condition \eqref{dfn:adapted-4}, 
note that there exists a meridian $\mu_C$ around each irreducible component $C$ of $F^*(P_k)$ of multiplicity $m=\mult\{C\}$, 
such that $F_*(\mu_C)=F_*(\gamma_k)^m$. Also note that
$$m_k=\gcd\{\mult(C)\in\ZZ_{\geq 1}\mid C \textrm{ irreducible component in } F^*(P_k)\}.$$
Using B\'ezout's identity one can obtain a product of meridians $\mu_k$ whose image is $F_*(\gamma_k)^{m_k}$.
In particular, $\mu_k=\alpha\gamma^{m_k}_k$ for $\alpha\in \ker F_*$. 
If $F^*(P)$ is not multiple, then $m_k=1$. Replacing $\gamma_k$ by $\alpha\gamma_k$ condition \eqref{dfn:adapted-4} follows, 
and conditions~\eqref{dfn:adapted-1}--\eqref{dfn:adapted-3} still hold.
\end{proof}

\begin{lemma}\label{lemma:semidirect}
Let $F:U\to S$ be an admissible map from a smooth quasi-projective surface $U$ to a smooth projective curve $S$ of genus $g_S$. Consider $B=\{P_0,P_1,\dots,P_n\}\subset S$, $n\geq 0$. Suppose moreover that $B\supset B_F$ contains the set $B_F$ of atypical values of $F$, and let $P\in S\setminus B$. 
Then $\pi_1(U_B)$ is a semidirect product of the form 
$$\pi_1(U_B)\cong \pi_1\left(F^{-1}(P)\right)\rtimes \pi_1\left(S\setminus B\right).$$
Moreover, $\pi_1(U_B)$ has a presentation with generators $\SF\cup \SB$ for 
$$
\SF=\{a_i,b_i,x_j\mid i=1,\ldots,g_F,\ j=0,\ldots,s_F-1\}, \textrm{ and }
\quad\quad\SB=\{\gamma_k\mid k=1,\ldots,r\},
$$
where $(g_F,s_F)$ is the type of $F^{-1}(P)$, $\SF$ (resp. $\SB$) is an adapted geometric set of fiber (resp. base) loops w.r.t. $F$ and $B$,
and the following is a set of relations
\begin{equation}
\label{eq:relations}
\array{l}
{[\gamma_k,a_i]=\alpha_{i,k}},\\
{[\gamma_k,b_i]=\beta_{i,k}},\\
{[\gamma_k\delta_{j,k},x_j]=1},\\
\prod_jx_j = \prod_i [a_i,b_i],
\endarray
\end{equation}
where $i\in\{1,\ldots,g_F\}$, $j\in\{0,\ldots,s_F-1\}$, $k\in\{1,\ldots,r=2g_S+n\}$, and $\alpha_{i,k}, \beta_{i,k}$, 
and $\delta_{j,k}$ are words in the elements of~$\SF$.
\end{lemma}

\begin{proof}
The condition $B\supset B_F$ implies that $F:U_B\rightarrow S\setminus B$ is a locally trivial fibration, 
with fiber $F^{-1}(P)$. Let $\iota:F^{-1}(P)\hookrightarrow U_B$ be the inclusion. The long exact sequence of 
a fibration for homotopy groups yields
	$$
	\pi_2(S\setminus B)\to \pi_1\left(F^{-1}(P)\right)\xrightarrow{\iota_*} \pi_1(U_B)\xrightarrow{F_*}  
	\pi_1\left(S\setminus B\right)\to 1.
	$$
	Note that, since $\pi_1\left(S\setminus B\right)\cong \FF_r$, the epimorphism $F_*$ splits.
	Since $S\setminus B$ is homotopy equivalent to a wedge of circles, $\pi_2(S\setminus B)=1$, 
	which concludes the result.

The description of the semidirect product $\pi_1\left(F^{-1}(P)\right)\rtimes \pi_1\left(S\setminus B\right)$ 
is given by considering the action of $\gamma_k$ on the group $\pi_1\left(F^{-1}(P)\right)$ generated by $\SF$.
For the generators $a_i$ (resp. $b_i$) one can write $\gamma_k^{-1}a_i\gamma_k=a_i\alpha_{i,k}$ 
($\gamma_k^{-1}b_i\gamma_k=b_i\alpha_{i,k}$) for some $\alpha_{j,k}$ (resp. $\beta_{j,k}$) in 
$\pi_1\left(F^{-1}(P)\right)$. For the meridians $x_j$ around the point $p_j$ on the boundary of $F^{-1}(P)$, 
note that $\gamma^{-1}_kx_j\gamma_k$ must be also a meridian around $p_j$ and hence 
$\gamma^{-1}_kx_j\gamma_k=\delta_{j,k}x_j\delta_{j,k}^{-1}$ for some $\delta_{j,k}$ in $\pi_1\left(F^{-1}(P)\right)$. 
The last relation in~\eqref{eq:relations} comes from the choice of the adapted geometric set of fiber loops~$\SF$.
\end{proof}

\begin{cor}
\label{cor:(0,1)}
Under the notation and assumptions of Lemma~\ref{lemma:semidirect}, $F_*:\pi_1(U_B)\to \FF_r$ is an isomorphism if and only if $F^{-1}(P)$ is of type~$(0,1)$ or $(0,0)$.

Moreover, if $F:\PP^2\dasharrow\PP^1$, and $U=\PP^2\setminus \mathcal{B}$, then $M_F=B_F$ and hence $\# B_F\leq 2$.
\end{cor}
\begin{proof}
The first statement is an immediate consequence of Lemma~\ref{lemma:semidirect}. If $U$ is a Zariski open 
subset of $\PP^2$, any pencil has at least a base point and thus any fiber $F^{-1}(P)$ of 
$F:\PP^2\dashrightarrow\PP^1$ must be an open curve, so the fibers of $F_{|U}$ will be open curves as well. 
The \emph{moreover} part follows from~\cite[Thm. 6.1]{Kashiwara-fonctions} by direct inspection, since all 
the pencils of type $(0,1)$ are classified. 
\end{proof}	

\begin{exam}
In particular, according to Corollary~\ref{cor:(0,1)}, the classification of all rational pencils on $\PP^2$ 
of type $(0,1)$ given in~\cite{Kashiwara-fonctions} provides a list of examples of curves whose complement have a 
free fundamental group.
\end{exam}

In Lemma~\ref{lemma:semidirect}, $B_F\subset B$. However, one can understand $\pi_1(U_B)$
for any non-empty finite set $B\subset S$ as follows.

\begin{cor}\label{cor:presentation}
Assume $F:U\to S$ is an admissible map, and let $B\subset S$, with $\# B=n+1\geq 1$. 
Then, $\pi_1(U_B)\cong\pi_1(U_{B\cup B_F})/N$, where $N$ is the normal closure of meridians 
$\gamma\in\pi_1(U_{B\cup B_F})$ of the components of $C_{B_F\setminus B}$.
\end{cor}
\begin{proof}
The result follows from Lemma~\ref{lemma:meridians-generators}.
\end{proof}

This result is well known in different settings
(cf. \cite{Nori-Zariski,Matsuno-Zariski,Shimada-Zariski}), but we include it here for the sake of completeness.

\begin{cor}
\label{cor:Fsurjective-orb}
Let $F:U\to S$ be an admissible map, let $B\subset S$ be a finite set with $\# B=n+1\geq 1$, and let 
$P\in S\setminus (B\cup B_F)$. 
Let $S_{(n+1,\bar m)}$ be the maximal orbifold structure on $S$ with respect to $F_{|{U_B}}$.
Then, 
$$
\pi_1(F^{-1}(P))\xrightarrow{\iota_*}\pi_1(U_B)\xrightarrow{F_*}\pi^{\orb}_1(S_{(n+1,\bar m)})\ \to 1
$$
is an exact sequence.
\end{cor}

 \begin{proof}
 Consider the commutative diagram:
 $$
 \begin{tikzcd}
 1\arrow[r] & \pi_1(F^{-1}(P))\arrow[r]\arrow[d,"\cong"] & \pi_1(U_{B\cup B_F}) \arrow[r, "F_*"]\arrow[d, two heads] & 
 \pi_1(S\setminus(B\cup B_F))\arrow[r]\arrow[d, two heads] & 1\\
 \ & \pi_1(F^{-1}(P))\arrow[r,"\iota_*"] & \pi_1(U_{B})\arrow[r,  "F_*"] & \pi_1^{\orb}(S_{(n+1,\bar m)})
 \end{tikzcd}
 $$
 Here, the vertical arrows are all induced by inclusion, and the top row is exact by
 Lemma~\ref{lemma:semidirect}.
 The surjectivity of $F_*$ in the bottom row follows from the diagram (but also from
 Remark~\ref{rem:inducedorb}).
 Also, $\im(\iota_*)\subseteq\ker F_*$. 
 Let us prove the other inclusion. Since $\im(\iota_*)$ is a quotient of a normal subgroup of 
 $\pi_1(U_{B\cup B_F})$, it is a normal subgroup. Using the exactness of the top row, the
 last paragraph of the proof of Lemma~\ref{lemma:admissible}, Lemma~\ref{lemma:semidirect} and
 Corollary~\ref{cor:presentation}, it is straightforward to see that the map
 $$
 \pi_1(U_{B})/\im(\iota_*)\twoheadrightarrow \pi_1^{\orb}(S_{(n+1,\bar m)})
 $$
 has a splitting $\sigma$ taking $F_*(\gamma_k)$ to the class of $\gamma_k$ for all
 $\gamma_k\in\Gamma_S$, where $\Gamma_S$ is an adapted geometric set of base loops with
 respect to $F$ and $B\cup B_F$. The commutativity of the diagram above implies that $\sigma$
 is surjective.
 \end{proof}

\begin{rem}\label{rem:Catanese-ses}
Suppose that $X$ is a projective surface, $S$ a smooth projective curve, and $F:X\to S$ is a
surjective holomorphic map with connected fibers. Let $S_{(0,\bar m)}$ be the maximal orbifold
structure of $S$ with respect to $F:X\to S$. As mentioned in the proof of
\cite[Lemma 4.2]{Catanese-Fibred}, one also has an exact sequence like the one in
Corollary~\ref{cor:Fsurjective-orb}, namely
$
\pi_1(F^{-1}(P))\to\pi_1(X)\to\pi_1^{\orb}(S_{(0,\bar m)})\to 1
$
where the first arrow is induced by the inclusion of a generic fiber $F^{-1}(P)$ over $P\in S$.
\end{rem}

\subsection{Characteristic Varieties}
Characteristic varieties are invariants of finitely presented groups $G$, and they can be computed using
any connected topological space $X$ (having the homotopy type of a finite CW-complex) such that 
$G=\pi_1(X,x_0)$, $x_0\in X$ as follows. Let us denote $H:=H_1(X;\ZZ)=G/G'$.
Note that the space of characters on $G$ is a complex torus
$
\TT_G:=\Hom(G,\CC^*)=\Hom(H,\CC^*)=H^1(X;\CC^*).
$
This $\TT_G$ can have multiple connected components, but it only contains one connected torus, which we
denote by $\TT_G^{\one}$.

\begin{dfn}
\label{def-char-var}
The $k$-th~\emph{characteristic variety}~of $G$ is defined by:
\[
\cV_{k}(G):=\{ \xi \in \TT_G\mid \dim H^1(X,\CC_{\xi}) \ge k \},
\]
where $H^1(X,\CC_{\xi})$ is classically called the 
\emph{twisted cohomology of $X$ with coefficients in the local system $\xi\in \TT_G$}.
It is also customary to use $\cV_{k}(X)$ for $\cV_{k}(G)$ whenever $\pi_1(X)=G$.
\end{dfn}

If $G=\FF_r*\ZZ_{m_1}*\dots*\ZZ_{m_s}$, the torus $\TT_G$ is a disjoint 
union of $\TT_G^{\one}\cong (\CC^*)^r$ and translations $\TT_G^{\lambda}$ of $\TT_G^{\one}$ by every 
element $\lambda$ of $C=C_{m_1}\times\ldots\times C_{m_s}$, where $C_{m}$ is the multiplicative group 
of $m$-th roots of unity. Given a torsion element $\rho\in \TT_G$ one can define
\begin{equation}
\label{eq:depth}
\depth(\rho):=\max\{k\in \ZZ_{\geq 0} \mid \TT_G^\rho\subset \cV_k(G)\}.
\end{equation}

\begin{rem}
\label{rem:general-type}
In~\cite[Prop. 2.10]{ACM-characteristic}, a complete description of $\cV_{k}(G)$ is given for 
orbifold fundamental groups of smooth quasi-projective curves. If $S=S_{(n+1,\bar m)}$ one can check that
$\cV_1(\pi_1^{\orb}(S))\neq\emptyset$ if and only if $\pi_1^{\orb}(S)$ is not abelian, or equivalently, 
if $\chi^{\orb}(S)<0$, in which case $S$ is called an \emph{orbifold of general type}.
\end{rem}

\subsection{Iitaka's (Quasi)-Albanese varieties}
Let $X$ be a smooth projective variety. The Albanese variety is defined as
\begin{equation*}
\Alb(X):= H^0(X,\Omega^1_X)^\vee/\mathrm{Free } H_1(X;\ZZ),
\end{equation*}
where $\vee$ denotes the dual as a $\CC$-vector space, and $\mathrm{Free } H_1(X;\ZZ)$ denotes
the torsion-free factor of $H_1(X;\ZZ)$. It is an abelian variety. Moreover, fixing a base point
$x_0\in X$, the Albanese morphism $\alpha_X:X \to\Alb(X)$ defined by
$x\mapsto \left(\omega\mapsto \int_{x_0}^x\omega\right)$ is an algebraic morphism.
Iitaka~\cite{Iitaka} generalized this to smooth quasi-projective varieties $U$ (for a detailed
description, see~\cite{Fujino-qa}), $\Alb(U)$ being a semiabelian variety in this case.
The Albanese map $\alpha_U$ satisfies that $(\alpha_U)_*:H_1(U;\ZZ)\to H_1(\Alb(U);\ZZ)$ is
surjective, whose kernel is $\mathrm{Tors }_{\ZZ} H_1(U;\ZZ)$. Moreover, if $X$ is a smooth
compactification of $U$ such that $D:=X\setminus U$ is a simple normal
crossing divisor, and $i:U\hookrightarrow X$ is the inclusion, then we have an exact sequence
\begin{equation}\label{eq:Albaneseses}
1\to (\CC^*)^r\to\Alb(U)\xrightarrow{\Alb(i)}\Alb(X)\to 1,
\end{equation}
where $r=\dim_{\CC} H^0(X,\Omega^1_X(\log D))-\dim_{\CC} H^0(X,\Omega^1_X)$.

We include here two technical lemmas about Albanese varieties.

\begin{lemma}\label{lemma:alb-admissible}
Let $U$ be a smooth quasi-projective surface such that $\pi_1(U)\cong\ZZ$.
Then, $\Alb(U)\cong \CC^*$, and $\alpha_U:U\to\CC^*$ is
an admissible map with no multiple fibers inducing isomorphisms in fundamental groups.
\end{lemma}

\begin{proof}
We have that $\Alb(U)\cong\CC^*$ because the latter is the only semiabelian variety
whose fundamental group is isomorphic to $\ZZ$. Note that $\alpha_U$ must be dominant.

Let us consider a holomorphic enlargement $F: X\to\PP^1$ of $\alpha_U$. Since $\pi_1(X)$ is a
quotient of $\ZZ$ by Lemma~\ref{lemma:meridians-generators}, using Stein factorization on $F$
and Remark~\ref{rem:h1-torsion}, we obtain that $\alpha_U$ factors through an admissible map
$f:U\to V\subset\PP^1$.
Thus,  $\pi_1(V)\cong\ZZ$, and thus $V=\CC^*$. 
By the universal property of the Albanese, $f:U\to\CC^*$ coincides with $\alpha_U$ up
to isomorphism of the target, so $\alpha_U$ is admissible.  Finally, by
Remark~\ref{rem:inducedorb}, $\alpha_U$ has no multiple fibers.
\end{proof}

\begin{lemma}\label{lemma:albanese}
Let $\hS$ be a smooth quasi-projective curve such that $\pi_1(\hS)\cong\FF_r$, for $r\geq 1$. 
Let $U$ be a smooth quasi-projective surface and $F:U\to \hS$ be an admissible map such that
$\overline F:X\to S$ is a holomorphic extension of $F$ with connected fibers, where $X$ is
a smooth projective compactification of $U$, $X\setminus U$ is a simple normal crossing
divisor and $S$ is a smooth projective curve of genus $g_{S}$. Let
$i_U:U\hookrightarrow X$ be the inclusion.

Suppose that $F_*:\pi_1(U)\to\pi_1(\hS)$ is an isomorphism.
\begin{enumerate}
\item If $g_{S}= 0$, then,
	\begin{enumerate}
	\item \label{lemma:albanese-eq0inj}$\alpha_{\hS}: \hS\to \Alb(\hS)$ is injective;
	\item \label{lemma:albanese-eq0torus}$\Alb(U)\cong (\CC^*)^r\cong \Alb(\hS)$;
	\item \label{lemma:albanese-eq0iso}$\Alb(F):\Alb(U)\to\Alb(\hS)$ is an isomorphism;
	\item \label{lemma:albanese-eq0} up to isomorphism in the target, the map $F$ coincides with the restriction of $\alpha_U$ 
	to its image, namely $\alpha_U: U\to \alpha_U(U)$. 
	\end{enumerate} 
\item If $g_{S}\geq 1$, then,
	\begin{enumerate}
	\item \label{lemma:albanese-geq1inj} $\alpha_{S}: S\to \Alb(S)$ is injective;
	\item \label{lemma:albanese-geq1iso} $\Alb(\overline F):\Alb(X)\to\Alb(S)$ is an isomorphism;
	\item \label{lemma:albanese-geq1} 
	up to isomorphism in the target, the map $F$ coincides with the restriction of $\alpha_X\circ i_U$ 
	to its image, namely $\alpha_X\circ i_U: U\to \alpha_X(U)$.
	\end{enumerate}
\end{enumerate}
\end{lemma}

 \begin{proof}
The variety $\Alb(S)$ has (complex) dimension $g_{S}$. Similarly, since $X$ is a projective variety,
the dimension of $\Alb(X)$ is half of the rank of $H_1(X,\ZZ)$. Let us show that the rank of $H_1(X,\ZZ)$ is $2g_S$.
 Applying Corollary~\ref{cor:Fsurjective-orb} to $F:U\to S'$, the morphism $\pi_1(\overline{F}^{-1}(P))\to\pi_1(X)$ induced by inclusion can be seen to be trivial, where $\overline{F}^{-1}(P)$ is a generic fiber of $\overline{F}$. Let $S_{(0,\bar m)}$ be the maximal orbifold structure of $S$ with 
 respect to $\overline{F}:X\to S$. By  Remark~\ref{rem:Catanese-ses}, $\overline{F}_*:\pi_1(X)\to\pi_1^{\orb}(S_{(0,\bar m)})$ 
 is an isomorphism. Abelianizing, we obtain that the rank of $H_1(X;\ZZ)$ is $2g_S$. In particular, if 
 $g_{S}=0$, then $\Alb(X)$ (resp. $\Alb(S=\PP^1)$) is a point, and thus, using equation \eqref{eq:Albaneseses},
 $\Alb(U)$ (resp. $\Alb(\hS)$) is a torus, which must necessarily have dimension $r$. This concludes the proof of part~\eqref{lemma:albanese-eq0torus}.
 
 Assume that $g_{S}=0$. Part~\eqref{lemma:albanese-eq0inj} is well known. Note that $\Alb(F):\Alb(U)\to\Alb(\hS)$ is an algebraic map which is a homomorphism between $(\CC^*)^r$ and itself 
 and induces an isomorphism on fundamental groups. 
 By Cartier duality (see~\cite{Russell-GeneralizedAlbanese}), $\Alb(F)$ is an isomorphism, and part~\eqref{lemma:albanese-eq0iso} is proved. Part~\eqref{lemma:albanese-eq0} now follows both from the functoriality of the Albanese map,
 and from parts~\eqref{lemma:albanese-eq0inj} and \eqref{lemma:albanese-eq0iso}.
 
 Assume now that $g_{S}\geq 1$. Part~\eqref{lemma:albanese-geq1inj} is the well known Abel-Jacobi theorem. 
 Let us prove part~\eqref{lemma:albanese-geq1iso}. Note that $\overline F:X\to S$ is surjective, so the classical Albanese map
 $\Alb(\overline F):\Alb(X)\to \Alb(S)$ is a surjective group homomorphism. Hence, 
 $\Alb(\overline F)$ must be a fibration, and the dimension of the fiber is $0$ when
 $\Alb(X)$ and $\Alb(S)$ have the same dimension, which we know equals $2g_{S}$ in both cases.
 Thus $\Alb(\overline F):\Alb(X)\to \Alb(S)$ is a finite covering. Since the fibers of
 $\overline F:X\to S$ are connected, the functoriality of the Albanese and
 Remark~\ref{rem:h1-torsion} imply that $\Alb(\overline F)$ is an isomorphism. In particular,
 $\alpha_{X}(X)$ is isomorphic to $S$. As in the $g_{S}=0$ case,
 part~\eqref{lemma:albanese-geq1} now follows both from the functoriality of the Albanese,
 and from parts~\eqref{lemma:albanese-geq1inj} and \eqref{lemma:albanese-geq1iso}.
 \end{proof}
 
\section{Main theorem}\label{sec:main}

Our purpose in this section is to give a necessary geometric condition for a curve to have
the fundamental group of its complement isomorphic to $\FF_r*\ZZ_{m_1}*\ldots*\ZZ_{m_s}$.
We will show that these curves come from admissible maps, with the only possible exceptions
occurring when the fundamental group is finite and the compact surface is not simply connected
(see Remark~\ref{rem:exception}).
We will prove the main theorem in two stages. If $r\geq 1$ this is done in
Theorem~\ref{lemma:existencePencil}.
If $r=0$ and the group is infinite, this is done in Theorem~\ref{thm:main2}. 
The strategy to find the admissible map in Theorem~\ref{thm:main2} heavily relies on the
structure of the fundamental group of $\pi_1(X\setminus D)$. The idea is to find a free
finite order subgroup whose associated covering falls in the hypotheses of
Theorem~\ref{lemma:existencePencil}.

\subsection{Proof of Theorem~\ref{thm:main:intro}}\label{sec:main:proof}
Theorem~\ref{thm:main:intro} will be stated in two ways depending on whether or not the
first Betti number of $X\setminus D$ vanishes.

\begin{thm}[Main theorem, $r\geq 1$]\label{lemma:existencePencil}
Theorem~\ref{thm:main:intro} holds if $r\geq 1$. 
\end{thm}
\begin{proof}
If $\pi_1(X\setminus D)\cong\ZZ$, ~\ref{lemma:existencePencil-2} and \ref{lemma:existencePencil-3} follow from
Lemma~\ref{lemma:alb-admissible} for $S_{(n+1,\bar m)}=\PP^1_{(2,-)}\cong\CC^*$.

Let us now prove ~\ref{lemma:existencePencil-2} and \ref{lemma:existencePencil-3} assuming that the group
$G:=\pi_1(X\setminus D)\cong\FF_{r}*\ZZ_{m_1}*...*\ZZ_{m_s}$ is non-abelian (i.e. $\not\cong\ZZ$). Using
Remark~\ref{rem:general-type} and~\cite[Prop. 2.10]{ACM-characteristic}, the $1$-st characteristic variety
of the group $G$ has a positive dimensional irreducible component $\TT_G^\lambda$ associated with the
torsion character $\lambda=(\xi_1,\dots,\xi_s)\in\TT_G$, where $\xi_i\in\CC$ is an $m_i$-th primitive
root of 1. Since $\TT_G^\lambda$ has dimension $r\geq 1$, by~\cite[Thm. 1]{ACM-characteristic}, there exists
an orbifold structure $S_{(n',\bar m')}$, $n'\geq 0$ on a smooth projective curve $S$ of genus $g_S$ and an
admissible map $F:X\dashrightarrow S$ which induces an orbifold morphism $F_|:X\setminus D\to S_{(n',\bar m')}$
such that $S_{(n',\bar m')}$ is maximal with respect to $F_|$ and $F_|^*(V_{\bar m'})=\TT_G^\lambda$ for some
component $V_{\bar m'}=\TT_{G_1}^{\lambda'}$ of the $1$-st characteristic variety of the orbifold fundamental
group $G_1=\pi_1^{\textrm{orb}}(S_{(n',\bar m')})$. Since $F$ has connected generic fibers, $F_|$ induces a
surjection of (orbifold) fundamental groups $F_*:G\to G_1$ (cf. \cite[Proposition 2.6]{ACM-characteristic})
and hence an injection $F^*:\mathcal{V}_k(G_1)\hookrightarrow \mathcal{V}_k(G)$ for all~$k$.

Since $\TT_G^{\lambda}$ is positive dimensional, it contains a non-torsion character.
By~\cite[Lemma 6.4]{ACM-characteristic}, the admissible map $F$ is unique such that
$F_|^*(\TT_{G_1}^{\lambda'})=\TT_G^{\lambda}$ and $\depth(\lambda)=\depth(\lambda')$.

Assume $n'=0$. According to the structure of its characteristic varieties
(cf. \cite[Proposition 2.11]{ACM-characteristic}),
one has $\dim \TT_{G_1}=2g_S-2=r-1=\dim \TT_{G}$. Finally, $1\in \mathcal{V}_{r+1}(G_1)$ but
$1\notin \mathcal{V}_{r+1}(G)$. This contradicts the inclusion of characteristic varieties for $k=r+1$.
Therefore $n'=n+1$, $n\geq 0$ and hence
\begin{equation*}
\pi_1^{\textrm{orb}}(S_{(n+1,\bar m')})\cong \FF_{r'}*\ZZ_{m_1'}*\ldots*\ZZ_{m'_{s'}},
\end{equation*}
where $n=r'-2g_S$. For $k=0$, $F_|^*(\TT_{G_1}^{\lambda'})=\TT_G^{\lambda}$ implies $r=r'$.

To show~\ref{lemma:existencePencil-2}, it remains to show that $s=s'$ and $m_i=m_i'$ for all $i=1,\ldots, s$.
Using~\cite[Prop. 2.10]{ACM-characteristic} and~\eqref{eq:depth}, one obtains
$s+r-1=\depth(\lambda)=\depth(\lambda')\leq s'+r-1$, so $s\leq s'$.
In addition, since $F_|^*$ induces injections between characteristic varieties,
one obtains $s'\leq s$, which shows $s=s'$.

Let $G=\langle x_1,\ldots,x_{r+s}\mid x_1^{m_1}=x_2^{m_2}=\ldots=x_s^{m_s}=1\rangle$ be the presentation of
$G$ that we have implicitly used to give coordinates to $\mathbb T_G$. Since $F_*$ is a surjection and the only
torsion elements of $G$ and $G_1$ are conjugation of elements in the finite group free factors,
Lemma~\ref{lemma:EpiImpliesIso} implies that, for some reordering of the $m_i'$, we can find a presentation
$\langle y_1,\ldots,y_{r+s} \mid y_1^{m_1'}=y_2^{m_2'}=\ldots=y_s^{m_s'}=1\rangle$ of $G_1$ such that
$F_*(x_j)=y_j$ for all $j=1,\ldots, r+s$. In particular, $m_i'\mid m_i$. We want to see that $m_i'=m_i$
for all $i=1,\ldots,s$. We argue by contradiction. Without loss of generality, we may assume that $m_1'<m_1$.
Hence, $F_{|}^*:\mathbb T_{G_1}\to \mathbb T_G$ takes a generator of the $\ZZ_{m_1'}$ factor of
$\mathbb T_{G_1}\cong \ZZ_{m_1'}\times\ldots\times \ZZ_{m_s'}\times (\CC^*)^r$ to an element of the subgroup
$\left\langle \frac{m_1}{m_1'}\right\rangle$ of the $\ZZ_{m_1}$ factor of
$\mathbb T_{G}\cong \ZZ_{m_1}\times\ldots\times \ZZ_{m_s}\times (\CC^*)^r$. In particular, $\lambda$ is
not in the image of $F_{|}^*$, which yields a contradiction. This concludes the proof
of~\ref{lemma:existencePencil-2} if $G$ is non-abelian.

The fact that $F_*$ is an isomorphism follows from Corollary~\ref{cor:Fsurjective-orb} and
Lemma~\ref{lemma:EpiImpliesIso}.
This concludes the proof of~\ref{lemma:existencePencil-3} if $G$ is non-abelian.

Let $B=\Sigma_0$ be the $n+1$ points of $S$ of label 0 in the orbifold structure $S_{(n+1,\bar m)}$.
Note that, by the maximality of the orbifold structure of $S_{(n+1,\bar m)}$ with respect to
$F_|:X\setminus D\to S_{(n+1,\bar m)}$, the extension
$F:X\dashrightarrow S$ satisfies that $\overline{F^{-1}(B)}=D_f\subset D$, hence $D=D_f\cup D_t$, where $D_t$
is the union of the irreducible components of $D$ which are not in $D_f$.  We can further decompose $D_t$ as
$D_v\cup D_h$, where $D_v$ is the union of the vertical components (irreducible components $C$ such that $F(C)$
is a point), and $D_h$ is the union of the horizontal components (irreducible components $C$ such that $F(C)=S$).
Note that we can choose a meridian around any irreducible component $C$ of $D_h$ which is contained in a generic
fiber of $F:X\setminus D\to S\setminus B$. Hence, by Corollary~\ref{cor:Fsurjective-orb}, any meridian about
any irreducible component of $D_h$ must be in the kernel
of $F_*$, thus it must be trivial as a consequence of part~\ref{lemma:existencePencil-3}.
Analogously, a meridian around an irreducible component of $D_v$ must be also in $\ker F_*$ since its image is a power
of a meridian around $P\in S\setminus B$ which is the boundary of a disk centered at $P$, and hence trivial.
This concludes~\ref{lemma:existencePencil-1}.
\end{proof}

\begin{thm}[Main theorem, $r=0$]
\label{thm:main2}
Theorem~\ref{thm:main:intro} holds if $r=0$ and $\pi_1(X\setminus D)$ is infinite. 
\end{thm}

\begin{proof}
Denote $G:=\pi_1(X\setminus D)=\ZZ_{m_1}*\ldots*\ZZ_{m_s}$, where $m_i>1$ for all $i\in I=\{1,\dots,s\}$ and $s\geq 2$.

Consider the unramified covering associated with the projection onto the maximal cyclic quotient $\nu:G\to \ZZ_{m}$,
where $m:=\lcm(m_i, i\in I)$. For any irreducible component $D_i$ of $D$, denote by $1\leq e_i\leq m$ the order of
$\nu(\gamma_i)\in\ZZ_m$, where $\gamma_i$ is any meridian around $D_i$.
By~\cite[Thm. 1.3.8.]{Namba-branched}, the cyclic unbranched covering $\theta:\tilde Y\to X\setminus D$
associated with $\nu:G\to \ZZ_{m}$ extends to a finite Galois branched cover $Y\to X$ which branches at each component $D_i$ with
branching number $e_i$. By~\cite[Theorem 1.1.7]{Namba-branched}, the deck action of $\ZZ_m$ on $\tilde Y$ extends to the
Galois action on $Y$, such that $Y/\ZZ_m\cong X$. Both of these actions are by algebraic automorphisms.

By Lemma~\ref{lemma:kernel}, the kernel $\ker \nu\cong\pi_1(\tilde Y)$ is the free group $\FF_{\rho}$, where
$\rho=1-m+m\sum_{i\in I}\left(1-\frac{1}{m_i}\right)\geq 1$.

Denote by $\bar Y$ a projective surface such that $\bar Y\setminus \bar D=U$, where $\bar D$ is a normal-crossing
divisor obtained by resolving the singularities of $D'$. We may assume that the action on $Y$ (which on $U$ is the action
by Deck transformations) extends to $\bar Y$. By Theorem~\ref{lemma:existencePencil} $U$ is induced by an admissible map
$f':U\to C'$ onto an (open) smooth curve $C'$ with no multiple fibers.
Moreover, by Lemma~\ref{lemma:albanese}\eqref{lemma:albanese-eq0} and \eqref{lemma:albanese-geq1},
up to an isomorphism of the curve, $f'$ is the restriction of the Albanese map to $U$ and its image.
By the functoriality of the Albanese map, the deck transformations can be carried over to $C'$.

Consider $\tilde\theta:C'\to C$ the quotient map by this action, where $C$ is an open curve (non necessarily smooth).
The map $f':\tilde Y\to C'$ hence descends to a morphism $f:X\setminus D\to C\subset \overline C$, where $\overline C$ is a
projective curve. Note that $C$ (and $\overline C$) may not be smooth, but, by the universal property of the
normalization, $f$ lifts to $\widetilde F:X\setminus D\to \widetilde S$, where $\widetilde S$ is the normalization
of $\overline C$ (a smooth projective curve). Applying Stein factorization, we know that $\widetilde F=\pi \circ F$,
where $F:X\setminus D\to S$ is admissible when restricted to its image, and $\pi:S\to \widetilde S$ is a finite morphism.

Since the normalization $\widetilde S\to \overline C$ is a birational equivalence, a generic fiber of $f$ is also a 
generic fiber of $\widetilde F$, which is a disjoint union of generic fibers of $F$. Moreover, since the generic 
fiber of $\widetilde\theta$ is finite, the preimage through $\theta$ of a generic fiber of $f$ is a disjoint union 
of generic fibers of $f'$. Restricting to connected components, one finds $P\in S$ and $P'\in C$ such that 
$\theta:(f')^{-1}(P')\to F^{-1}(P)$ is a finite covering map, where $(f')^{-1}(P')$ is a generic fiber of $f'$ 
and $F^{-1}(P)$ is a generic fiber of $F$.

Let us check that the morphism $F_*:\pi_1(X\setminus D)\to \pi_1^{\orb}(S)$ is an isomorphism, where $S$ is endowed
with the maximal orbifold structure with respect to $F:X\setminus D\to S$. By Corollary~\ref{cor:Fsurjective-orb}, 
this is equivalent to showing that the image of $\pi_1(F^{-1}(P))$  in $\pi_1(X\setminus D)$ is trivial. Since 
$f':U\to C'$ induces an isomorphism on fundamental groups, Corollary~\ref{cor:Fsurjective-orb} tells us that the 
inclusion $(f')^{-1}(P')\hookrightarrow U$ induces the trivial map on fundamental groups. Consider the commutative diagram
$$
\begin{tikzcd}
\pi_1((f')^{-1}(P))\arrow[r]\arrow[d, "\theta_*"] & \pi_1(U)\arrow[d, "\theta_*"]&\\
\pi_1(F^{-1}(P))\arrow[r, "i_*"]&\pi_1(X\setminus D)&\hspace*{-1cm}\cong\ZZ_{m_1}*\ldots*\ZZ_{m_s},
\end{tikzcd}
$$
where the horizontal arrows are induced by inclusion. Note that the arrow on the left is a
finite covering space, so its image is a finite index normal subgroup of $\pi_1(F^{-1}(P))$.
The commutativity of the diagram above implies that the morphism
$i_*$ factors through the quotient $\pi_1(F^{-1}(P))/\theta_*\left(\pi_1((f')^{-1}(P))\right)$.
Hence, the image of $i_*$ is a finite subgroup of $\ZZ_{m_1}*\ldots*\ZZ_{m_s}$. Moreover, by Corollary~\ref{cor:Fsurjective-orb}, this subgroup is
normal. By the Kurosh subgroup theorem, the only finite normal subgroup of $\ZZ_{m_1}*\ldots*\ZZ_{m_s}$ is the trivial
subgroup. Hence, $i_*$ is the trivial morphism, and thus $F_*:\pi_1(X\setminus D)\to \pi_1^{\orb}(S)$ is an isomorphism.

Note that if $S$ is a smooth projective curve such that $\pi_1^{\orb}(S_{(n,\bar m)})\cong\ZZ_{m_1}*\ldots*\ZZ_{m_s}$,
then we will see that $S=\PP^1$, and the image of $F$ must be $\PP^1$ with one point removed. Indeed, after
abelianizing the presentation of $\pi_1^{\orb}(S)$ from Section~\ref{sec:orbifold}, it follows that $S$ must have
genus $0$ (so $S=\PP^1$).
Moreover, $F(X\setminus D)$ is either $\PP^1$ or $\CC$. Let us see that it is indeed the latter. Suppose that
$F(X\setminus D)=\PP^1$, so none of the irreducible components of $D$ are fibers of $F:X\dashrightarrow \PP^1$. 
Then, as in Theorem~\ref{lemma:existencePencil}, the inclusion of $X\setminus D$ to $X$ induces an isomorphism in 
fundamental groups, and thus $\pi_1(X)$ is isomorphic to a non-trivial free product. This is impossible 
by~\cite{Gromov-fundamental}.

We have shown that, under the assumptions of this theorem, if $\pi_1(X\setminus D)\cong\ZZ_{m_1}*\ldots*\ZZ_{m_s}$, 
with $m_i,s\geq 2$ for all $i=1,\ldots, s$, then, there exists an admissible map $F:X\setminus D\to \PP^1\setminus{B}$ 
that induces an isomorphism $F_*:\pi_1(X\setminus D)\to\pi_1^{\orb}(\PP^1_{(1,\bar m)})$, where $\PP^1$ is endowed with 
the maximal orbifold structure with respect to $F:X\setminus D\to\PP^1$. Note that $\bar m=(m_1,\ldots,m_s)$, and those 
are the multiplicities of the multiple fibers. The remaining condition for $D=D_f\cup D_t$ can be proved as 
in~\ref{lemma:existencePencil}.
\end{proof}

\begin{rem}
The mere existence of zero-dimensional components of the characteristic varieties of 
$\pi_1(X\setminus D)$ is not enough to ensure the existence of the admissible map $F$.
For instance, in~\cite[Thm. 4.5]{ACM-multiple-fibers} an example of a quintic projective curve $D$
considered by Degtyarev is presented, where the characteristic varieties of $\pi_1(\PP^2\setminus D)$ 
are the primitive 10th roots of unity, but no admissible map $\PP^2\dasharrow\PP^1$ exists inducing
a surjection from $\pi_1(X\setminus D)$ onto $\pi_1^{\orb}(\PP^1_{1,\bar m})$.
\end{rem}

\begin{exam}
\label{exam:aguilar}
Examples of smooth quasi-projective surfaces having infinite fundamental groups of the form 
$\ZZ_{m_1}*\dots*\ZZ_{m_s}$ were found by Aguilar Aguilar in~\cite[Thm. 1.2]{Aguilar-fundamentalgroup}.
The author considers three concurrent lines $L+L_1+L_2$ in $\PP^2$ intersecting at $P$ and blows up $P$ and the 
infinitely near points $P_i$ on the strict transform $\tilde L_i$ of $L_i$, $i=1,2$. 
Hence $\pi^*(L_i)=E+2E_i+\tilde L_i$, where $E$ is the exceptional divisor of the first blow up and $E_i$ is 
the exceptional divisor that appears when blowing up the infinitely near points $P_i$. Denote by $\hat \PP^2$ the 
resulting surface after the three blow-ups. The quasi-projective surface 
$X=\hat \PP^2\setminus (\tilde L\cup\tilde L_1\cup \tilde L_2)$ has a well-defined morphism onto the orbifold 
$\PP^1_{(1,(2,2))}$ whose generic fiber is the strict transform of a generic line through $P$, which is a rational
smooth curve and hence simply connected. Hence $\pi_1(X)=\ZZ_2*\ZZ_2$. The author shared with us how his method can 
be generalized by considering $r+s+1$ concurrent lines and blowing up $m_i$ times at infinitely near points of $P$ 
in $L_i$, $i=1,\dots,s$ so that $\pi^*(L_i)=E+2E_{i,2}+\dots+m_iE_{i,m_i}+\tilde L_i$, to produce surfaces 
\begin{equation*}
X=\hat \PP^2\setminus \left(\bigcup_{i=1}^{r+s+1} \tilde L_{i}\cup
\bigcup_{i=1,\dots,s}
\bigcup_{k_i=1}^{m_i-1} E_{i,k_i}\right),
\end{equation*}
with $\pi_1(X)=\FF_r*\ZZ_{m_1}*\dots*\ZZ_{m_s}$ for any $r\geq 0$, $m_1,\ldots,m_s\geq 0$. In particular, every
finitely generated free product of cyclic groups can be realized as the fundamental group of a smooth
quasi-projective surface.
\end{exam}

We now make an observation about the type of free products of cyclic groups that can appear as 
fundamental groups of curve complements in \textit{simply connected} projective surfaces.

\begin{rem}
\label{rem:2fibers-bis}
Under the conditions of Theorem~\ref{thm:main:intro}, if $X$ is simply connected, 
then $s\leq 2$ (see Remark~\ref{rem:2fibers}) that is,
$\pi_1^{\orb}(\PP^1_{(r+1,\bar m)})\cong\FF_r*\ZZ_p*\ZZ_q$, where $p,q\geq 1$. 
Moreover, if $D$ satisfies Condition~\ref{eq:numequiv}, then $\gcd(p,q)=1$ by Lemma~\ref{lem:cyclichomology}.
\end{rem}

\begin{exam}
The simply-connectedness condition in Remark~\ref{rem:2fibers-bis} is important. We will see an example satisfying 
the hypotheses of Theorem~\ref{thm:main2} with a non simply-connected surface $X$, a curve $D$ on $X$ such that 
$\pi_1(X\setminus D)\cong \ZZ_3*\ZZ_3$, and a rational map realizing the isomorphism with the orbifold 
fundamental group of~$\PP^1_{(1,(3,3))}$. 
Let $X$ be the double cover of $\PP^2$ ramified along a generic sextic $C=\{f_2^3+f_3^2=0\}$, where $f_i$ is 
a homogeneous polynomial in three variables of degree~$i$. It is well known (cf.~\cite{Zariski-problem}) 
that $\pi_1(\PP^2\setminus C)=\ZZ_2*\ZZ_3$
(also as a consequence of Theorem~\ref{thm:Okapq}). A meridian around $C$ can be given 
as $xy$ where $x^2=y^3=1$. Using Lemma~\ref{lemma:kernel}
one can show $\pi_1(X\setminus D)\cong \ZZ_3*\ZZ_3$ generated by $y_0:=y$ and $y_1:=xyx$, 
where $D$ is the preimage of~$C$. Note that the six-fold cover ramifies fully along $C$ and thus its preimage is irreducible.
Since this cover factors through $X$, $D$ must also be irreducible and thus it satisfies Condition~\ref{eq:numequiv}.
The preimage of $(yx)^2=y_0y_1$ becomes a meridian of $D$ and by Lemma~\ref{lemma:meridians-generators} one has
$\pi_1(X)=\pi_1(X\setminus D)/\ncl(y_0y_1)=H_1(X;\ZZ)=\ZZ_3$. Theorem~\ref{thm:main2} ensures the existence 
of an admissible map $F:X\dashrightarrow \PP^1$ with two multiple fibers of multiplicity 3 and such that 
$D$ is the preimage of a point in~$\PP^1$.
\end{exam}

\begin{cor}\label{cor:rD}
Under the notation of Theorem~\ref{thm:main:intro}, suppose that an admissible map $F:X\dasharrow S$ satisfies conditions~\ref{lemma:existencePencil-2}--~\ref{lemma:existencePencil-1}, so in particular $\pi_1(X\setminus D)$ is isomorphic through the map induced by $F$ to an open orbifold group of $S$ (not necessarily infinite). If $H_1(X;\ZZ)$ 
is torsion and $D$ satisfies Condition~\ref{eq:numequiv}, then $S=\PP^1$ and $r+1$ 
is in fact the number of irreducible components of~$D$. In particular, $D=D_f$ is a fiber-type curve.
\end{cor}

\begin{proof}
The result follows from Remark~\ref{rem:h1-torsion} and Lemma~\ref{lem:cyclichomology}.
\end{proof}

The following example will exhibit the importance of Condition~\ref{eq:numequiv} to establish that $r+1$ is the 
number of irreducible components of~$D$.

\begin{exam}
Let $\cQ$ be a smooth conic in $\PP^2$ and $P\in\cQ$. Consider $\ell$ the tangent line to $\cQ$ through $P$ and 
$D$ a union of $\ell$ and $r$ lines through $P$. The quadric surface $X=\PP^1\times\PP^1$ can be defined as the 2:1 
cover $\sigma:X\to\PP^2$ ramified along $\cQ$. In particular, $X$ is simply connected. Also note that $\sigma^{-1}(D)$ 
is a union of $r+2$ irreducible components, namely $r$ curves of bidegree $(1,1)$ and the rulings 
$\sigma^{-1}(\ell)=\ell_1\cup\ell_2$,
for $\ell_1$ (resp. $\ell_2$) of bidegree $(1,0)$ (resp. $(0,1)$) all of them passing through $\sigma^{-1}(P)$.
The curve $\sigma^{-1}(D)$ has $r+2$ irreducible components and does not satisfy Condition~\ref{eq:numequiv}.
Note that $\pi_1(\PP^2\setminus D)=\FF_r$ and one can check that also $\pi_1(X\setminus \sigma^{-1}(D))=\FF_r$.
\end{exam}

Theorem~\ref{thm:main:intro} gives necessary geometric conditions for a quasi-projective surface $X\setminus D$ 
to have an infinite fundamental group which is a free product of cyclic groups. The curve $D$ need not be of 
fiber-type, but $D_f$ (which is a non-empty union of irreducible components of $D$) is a fiber-type curve coming from an admissible
map $F:X\dasharrow S$. The following result illustrates that
$X\setminus D_f$ behaves exactly like $X\setminus D$ in Theorem~\ref{thm:main:intro}.
 
\begin{cor}\label{cor:non-fiber-to-fiber}
Under the conditions of Theorem~\ref{thm:main:intro} and using the notation therein, the inclusion induces an 
isomorphism $\pi_1(X\setminus D)\cong\pi_1(X\setminus D_f)$, and
$F_*:\pi_1(X\setminus D_f)\to\pi_1^{\orb}(S_{(n+1,\bar m)})$ is an isomorphism.
\end{cor}

\begin{proof}
The isomorphism $\pi_1(X\setminus D)\cong\pi_1(X\setminus D_f)$ follows from Lemma~\ref{lemma:meridians-generators} and the 
fact that the meridians of $D_t$ are trivial in $\pi_1(X\setminus D)$. Using Corollary~\ref{cor:Fsurjective-orb} for 
$U=X\setminus (D_t\cup\mathcal B)$ and $U=X\setminus \mathcal B$, we see that $F_|:X\setminus D_f\to S\setminus \Sigma_0$ must also 
induce isomorphisms in (orbifold) fundamental groups, and that the maximal orbifold structure on $S$ with respect to $F_|:X\setminus D_f\to S\setminus \Sigma_0$ 
must coincide with the one with respect to $F_|:X\setminus D\to S\setminus \Sigma_0$.
\end{proof}

\subsection{Extensions of the main theorem to finite cyclic groups}\label{sec:main:extensions}
\mbox{}
Theorem~\ref{thm:main:intro} describes the geometry of a curve $D\subset X$ when 
$\pi_1(X\setminus D)$ is an \textit{infinite} group of the form $\FF_r*\ZZ_{m_1}*\ldots*\ZZ_{m_s}$ (i.e. non-abelian or $\ZZ$), where $X$ is a 
smooth projective surface. In this section, we give extra hypotheses under which similar results hold in 
the remaining abelian cases (the trivial group and finite cyclic groups).

\begin{prop}[Main theorem, trivial group case]\label{prop:maintrivial}
Let $D\subset X$ be a curve in a smooth projective surface $X$. Assume that $\pi_1(X\setminus D)$ is trivial. 
If $D$ is an ample divisor, then there exists an admissible map $F:X\dashrightarrow \PP^1$ as in 
Theorem~\ref{thm:main:intro} for $S_{(n+1,\bar m)}=\PP^1_{(1,-)}$ satisfying 
conditions~\ref{lemma:existencePencil-2}-\ref{lemma:existencePencil-1}.
\end{prop}

\begin{proof}
Since $D$ is ample, $D$  defines an embedding $X\setminus D\hookrightarrow\CC^k$ for some $k$. Projecting to a 
generic $1$-dimensional subspace inside $\CC^k$, we get a dominant morphism $F:X\setminus D\to \CC$ with connected 
fibers, which can be extended to a rational map $F:X\dashrightarrow \PP^1$. Using Remark~\ref{rem:inducedorb} we 
have that $\pi_1^{\orb}(\PP^1_{(n+1,\bar m)})$ is trivial, where $\PP^1_{(n+1,\bar m)}$ is the maximal orbifold 
structure with respect to $F:X\setminus D\to\CC$. This implies that $n=0$ (so in particular $F:X\setminus D\to\CC$ is surjective) 
and that $\bar m$ is the trivial orbifold structure.
\end{proof}

\begin{rem}
In order to clarify the hypothesis given in Proposition~\ref{prop:maintrivial} we will exhibit an example where 
the result does not follow when $\pi_1(X\setminus D)=1$ and $D$ is not an ample divisor.

Consider a line $L$ in a smooth cubic $X$. It is well known that $\pi_1(X)=\pi_1(X\setminus L)=1$ 
and $L^2=-1$. Thus $L$ cannot be the fiber of an admissible map $X\dashrightarrow\PP^1$.
\end{rem}

\begin{prop}[Main theorem, case $\ZZ_m$, $m> 1$]\label{prop:mainfinite}
Let $D\subset X$ be a curve in a smooth projective surface $X$. Assume that $\pi_1(X\setminus D)\cong{\ZZ_m}$, 
for $m>1$. If $X$ is simply connected, then
there exists an admissible map $F:X\dashrightarrow \PP^1$ as in Theorem~\ref{thm:main:intro}
for $S_{(n+1,\bar m)}=\PP^1_{(1,(m))}$ satisfying conditions~\ref{lemma:existencePencil-2}-\ref{lemma:existencePencil-1}. 
Moreover, $D=D_f$ is a fiber-type curve.
\end{prop}

\begin{proof}
Let $D=\cup_{i=0}^r D_i$ be the decomposition of $D$ into irreducible components. Any meridian $\gamma_i$ around an 
irreducible component $D_i$ of $D$ has finite order $e_i$ dividing $m$. Consider the divisor $E=\sum \frac{m}{e_i}D_i$. Since $\pi_1(X)$ is 
the result of factoring out $\pi_1(X\setminus D)$ by the normal closure of the meridians around all the $D_i$'s, 
and $X$ is simply connected, note that $E$ is not a positive multiple of an effective divisor. 
By~\cite[Thm. 1.3.8.]{Namba-branched} the unbranched universal covering $\tilde Y\to X\setminus D$ 
associated with $G$ extends to a branched covering $Y\to X$. This implies there exists an effective divisor $H$ in $X$ 
such that $E\sim mH$. The linear equivalence provides a morphism $F:X\dashrightarrow \PP^1$ such that 
$F_|:X\setminus D\rightarrow \CC$ (so in particular $D$ is a fiber-type curve). After applying Stein factorization, we may assume that $F:X\dashrightarrow \PP^1$ is the 
composition of $\widetilde F:X\dashrightarrow\PP^1$ and $\beta:\PP^1\to\PP^1$, where $\widetilde F$ has connected fibers and 
$\beta$ is generically $k:1$. In principle, $D$ is a union of $n+1$ fibers of $\widetilde F$ above points in $\PP^1$, 
although it must be only one fiber, or else $\widetilde F_*:\pi_1(X\setminus D)\to\pi_1^{\orb}(\PP^1_{(n+1,\bar m)})$ would not
be surjective, which would contradict Remark~\ref{rem:inducedorb}. Now, $F^*([1:0])=E$ and $\beta^{-1}([1:0])$ is just a point, so $E$ must be $k$ times an effective 
divisor with support $D$, and hence $k=1$. Thus, $F:X\dashrightarrow \PP^1$ has connected fibers. We know that $F^*([0:1])=mH$, 
so $F$ induces a surjective morphism $\pi_1(X\setminus D)\to\pi_1^{\orb}(\PP^1_{(1,(m))})$. Since $\ZZ_m$ is finite, this
morphism is an isomorphism, and $\PP^1_{(1,(m))}$ is the maximal orbifold structure with respect to $F_|:X\setminus D\to\CC$.
\end{proof}

\begin{rem}
\label{rem:exception}
In order to clarify the simply-connected hypothesis given in Proposition~\ref{prop:mainfinite} we will exhibit an 
example where the result does not follow when $\pi_1(X)\neq 1$. Consider a surface $X\subset \PP^n$ with finite 
cyclic fundamental group $\ZZ_m$ and containing a line $L\subset X$. Take a hyperplane $H\subset \PP^n$ intersecting 
$X$ transversally such that $L\not\subset H$. Note that $D=H\cdot X$ defines a reduced irreducible divisor in $X$ and 
$L\cap D=L\cap H=\{P\}$. Hence, one can define a meridian $\gamma$ of $D$ around $P$ such that $\gamma\subset L$. 
Since $L$ is a rational curve, $\gamma$ is trivial in $X$. If there was a map $F:X\dashrightarrow S$ such that 
$D$ is a fiber of $F$, then $F_*:\pi_1(X\setminus D)\to \pi^{\orb}_1(S\setminus \{p\})$ would be surjective.
Since $\pi_1(X\setminus D)=\ZZ_m$, this implies $g_S=0$ and the orbifold structure of $S$ contains exactly one 
orbifold point of order $m>1$. However, in that case $F_*(\gamma)$ would have order $m$, which is a contradiction 
since we have proved that $\gamma$ is trivial.
\end{rem}

\subsection{Main Theorem for curves in $\PP^2$}\label{sec:main:P2}
\mbox{}
We will pay a special attention to the case $X=\PP^2$, in Theorem~\ref{thm:main:intro}.
Recall that every curve in $\PP^2$ satisfies Condition~\ref{eq:numequiv} and is ample. Hence, the extra hypotheses 
needed in the relevant results of Sections~\ref{sec:main:proof} and~\ref{sec:main:extensions} are always satisfied for curves in $\PP^2$, and
thus, the following stronger version of the Main Theorem~\ref{thm:main:intro} holds.

\begin{cor}[Main Theorem, curves in $\PP^2$]\label{thm:main:P2}
Let $D$ be a curve in $\PP^2$.
Suppose that $\pi_1(\PP^2\setminus D)$ is a free product of cyclic groups. 
Then, there exist $r\geq 0$ and $p\geq q\geq 1$ with $\gcd(p,q)=1$ such that $\pi_1(\PP^2\setminus D)\cong\FF_r*\ZZ_p*\ZZ_q$,
and an admissible map $F:\PP^2\dashrightarrow \PP^1$ such that:
\begin{enumerate}
\enet{\rm(\roman{enumi})}
\item \label{thm:main-simply-1}
$F$ induces an orbifold morphism
$$
F_|:\PP^2\setminus D\to \PP^1_{(r+1,\bar m)},
$$
where $\PP^1_{(r+1,\bar m)}$ is maximal with respect to $F_|$, and $\bar m=(p,q)$.
\item \label{thm:main-simply-2} 
$F_*:\pi_1(\PP^2\setminus D)\to \pi_1^{\orb}(\PP^1_{(r+1,\bar m)})$ is an isomorphism.
\item \label{thm:main-simply-3}
$D=\overline{F^{-1}(\Sigma_0)}$ is a fiber-type curve which is the union of $r+1$ irreducible fibers of~$F$.
\end{enumerate}
Moreover, after possibly a change of coordinates in $\PP^1$, one has $F=[f_{d_2}^{d_1}:f_{d_1}^{d_2}]$,
where $d_1\geq d_2\geq 1$ satisfy that $\gcd(d_1,d_2)=1$, $f_{d_i}$ is a homogeneous polynomial of degree $d_i$ for $i=1,2$ which is not a $k$-th power of another polynomial for any $k\geq 2$, $f_{d_1}$ and $f_{d_2}$ do not have any components in common, and $M_F\subset\{[0:1],[1:0]\}$. More concretely,
\begin{enumerate}
\item\label{thm:main-simply-4-1} If $p>q>1$, then $d_1=p$, $d_2=q$ and the pencil $F=[f_{q}^p:f_{p}^q]$ 
has exactly two multiple fibers corresponding to $[0:1],[1:0]\notin \Sigma_0$.
\item\label{thm:main-simply-4-2} 
If $p>q=1$, then $d_1=p$, $[1:0]\in \Sigma_0$, and the pencil $F=[f_{d_2}^p:f_p^{d_2}]$ 
has at least one multiple fiber corresponding to $[0:1]\notin \Sigma_0$.
\item\label{thm:main-simply-4-3} 
If $q=p=1$, then $F=[f_{d_2}^{d_1}:f_{d_1}^{d_2}]$ has at most two multiple 
fibers corresponding to $[0:1],[1:0]\in \Sigma_0$.
\end{enumerate}
\end{cor}

\begin{proof}
The existence of an admissible map $F:\PP^2\dasharrow \PP^1$ inducing an isomorphism 
$F_*:\pi_1(\PP^2\setminus D)\to\pi_1^{\orb}(\PP^1_{(r+1,\bar m)})$ follows from Theorem~\ref{thm:main:intro} and Remark~\ref{rem:h1-torsion}, as well as from Propositions~\ref{prop:maintrivial} and \ref{prop:mainfinite}. By Remark~\ref{rem:2fibers-bis}, 
$\pi_1(X\setminus D)\cong\FF_r*\ZZ_p*\ZZ_q$. This concludes the proof of parts~\ref{thm:main-simply-1} and \ref{thm:main-simply-2}. Part~\ref{thm:main-simply-3} follows from Corollary~\ref{cor:rD}.

By Remark~\ref{rem:2fibers}, after possibly a change of coordinates in $\PP^1$, we may assume $M_F\subset\{[0:1],[1:0]\}$. $F$ is of the form $F=[f_{k_2}^{d_1}:f_{k_1}^{d_2}]$ with  where $f_{k_i}$ is a homogeneous polynomial of degree $k_i$ such that $d_1k_2=d_2k_1$, $f_{k_1}$ and $f_{k_2}$ have no common factors, and $f_{k_i}$ is not a $l$-th power of another polynomial for any $l\geq 2$, for $i=1,2$.
The condition $\gcd(d_1,d_2)=1$ is necessary for the fibers of $F$ to be connected. In particular, there exists $l\in\ZZ_{>0}$ such that $k_i=d_il$ for $i=1,2$. The rest of the proof is straightforward 
using Remark~\ref{rem:coprime}, including the fact that $l=1$ (and hence $d_i=k_i$ for $i=1,2$) in each case \eqref{thm:main-simply-4-1}--\eqref{thm:main-simply-4-3}.
\end{proof}

\begin{rem}
In a series of papers, Eyral-Oka calculated the fundamental group of the
affine complement of a curve of type $g(x)=f(y)$ under certain conditions, as well as the fundamental group of the projective complement of the corresponding projectivized curve  (see~\cite{Oka-genericRjoin} for the case where the curve of that type is generic, see~\cite{Eyral-Oka-RjoinType} for the semi-generic case, and see~\cite{Eyral-Oka-RjoinTypeII} for curves satisfying a certain arithmetic condition).

Under their conditions, the fundamental group of the complement in $\PP^2$ of the projectivization of such a curve
is a free product of cyclic groups if $d=\deg(f)=\deg(g)$. Note that, in this case, the projectivized polynomials
$\overline g(x,z)$ and $\overline f(y,z)$ are of the form $f_q^{rp}$ and $f_p^{rq}$ for some $p$ and $q$ coprime
and $r\geq 1$,
where $f_p$ and $f_q$ are homogeneous polynomials of degree $p$ and $q$ respectively which are not a $k$-th power of any
other polynomial for $k\geq 2$ (in fact, $f_p$ and $f_q$ are products of powers of degree $1$ polynomials). Hence, in
that case, the projective curve $\overline g(x,z)=\overline f(y,z)$ is a union of $r$ fibers of an admissible map
$F=[f_p^q:f_q^p]:\PP^2\setminus \cB\to \PP^1$. Moreover, the fundamental group of its complement in $\mathbb{P}^2$ is
$\ZZ_p*\ZZ_q*\FF_{r-1}$. This agrees with Corollary~\ref{thm:main:P2}.
\end{rem}

\section{Addition-Deletion Lemmas}\label{sec:additiondeletion}
In Section~\ref{sec:main}, we have seen geometric conditions for a quasi-projective surface $X\setminus D$ to have a free product of cyclic groups as its fundamental group. Under the hypotheses  
and notation of Theorem~\ref{thm:main:intro}, $D=D_f\cup D_t$. 
Hence, if we let $U=X\setminus(D_t \cup \mathcal B)$, then $X\setminus D=U\setminus D_f$ is the complement 
of a fiber-type curve inside a smooth quasi-projective surface given by an admissible map $F:U\to S$ onto a smooth 
projective curve~$S$.

The purpose of this section is to prove addition-deletion results of fibers for complements of fiber-type curves 
inside smooth quasi-projective surfaces $U$ whose fundamental group is isomorphic by the map induced by $F$ to an open orbifold fundamental group of $S$. Before we do that, we need 
some technical results regarding presentations of fundamental groups of complements of fiber-type curves.

\subsection{Preparation Lemmas}
The results in this section give presentations of fundamental groups of fiber-type curve complements, but they 
do not make use of any Zariski-Van Kampen type computations. Instead, they follow from group-theoretical arguments and from
Lemma~\ref{lemma:semidirect}, Corollaries~\ref{cor:presentation} and~\ref{cor:Fsurjective-orb}, and the results in
Section~\ref{sec:fundamentalgroups}.

Assume $F:U\to S$ is an admissible map and consider $M_F\subset S$ (resp. $B_F$) the set of multiple (resp. atypical)
values of $F$ (see section~\ref{sec:admissible}). In Lemma~\ref{lemma:semidirect}, we gave a presentation of $\pi_1(U_B)$ in the
case where $B\supset B_F$. We now give a more explicit presentation for $\pi_1(U_B)$ in the general case when $B$ does
not necessarily contain~$B_F$.

Throughout this section, $F:U\to S$ will be an admissible map from a smooth quasi-projective surface $U$ to a smooth 
projective curve~$S$. We will use the notation introduced in Section~\ref{sec:pi1-fiber-type}.

\begin{lemma}\label{lem:presentationexplicit}
Assume that $F:U\to S$ is an admissible map. Let $B\subset S$, with $\# B\geq 1$.
Consider $\SF$ (resp. $\SB=\SB(B\cup B_F)$) an adapted geometric set of fiber (resp. base) loops
w.r.t. $F$ and $B$ (resp. $B\cup B_F$). Let $B'=M_F\setminus B$, and let $\SBp=\SB\setminus\{\gamma_k\mid P_k\in B_F\setminus (B\cup B')\}$.

Then, $\SBp$ is an adapted geometric set of base loops w.r.t. $F$ and $B\cup B'$, and $\pi_1(U_B)$ has a finite presentation with generators $\SF\cup\SBp$ and relations of the form 

\begin{enumerate}
\enet{\rm(R\arabic{enumi})}
 \item\label{R1}
 $[\gamma_{k},w]=1$, for any $w\in \SF$ and any $\gamma_k\in \SBp$ a lift of a meridian of $P_k\in B\setminus B_F$,
 \item\label{R2}
 $[\gamma_{k},w]=z_{k,w}$, for the remaining $\gamma_k\in \SBp$, for any $w\in \SF$, where $z_{k,w}$ is a word in~$\SF$,
 (that depends both on $\gamma_k$ and $w$),
 \item\label{R3}
 $y=1$ for a finite number of words $y$ in $\SF$,
 \item\label{R4} 
 $z_{k}=\gamma_{k}^{m_{k}}$, for any $\gamma_k\in \SBp$ a lift of a meridian of $P_k\in B'$, where $m_k$ is the
 multiplicity of the fiber $F^*(P_k)$, and $z_{k}$ is a word in~$\SF$,
\end{enumerate}
\end{lemma}
\begin{proof}
Note that $B\cup B_F$ satisfies the conditions of Lemma~\ref{lemma:semidirect}, and hence $\pi_1(U_{B\cup B_F})$ admits
a presentation generated by $\SF\cup\SB$, where $\SF$ (resp. $\SB$) is an adapted geometric set of fiber (resp. base) loops w.r.t. $F$ and $B\cup B_F$. Note that $\SF$ is an adapted geometric set of fiber loops w.r.t. $F$ and $B\cup B_F$ if and only if the same holds w.r.t. $B$.

By Corollary~\ref{cor:presentation} a presentation of $\pi_1(U_B)$ can be obtained by factoring $\pi_1(U_{B\cup B_F})$ 
by the normal closure of all the meridians about irreducible components of fibers above points in $B_F\setminus B$. 
We abuse notation and see both the elements of $\SF$ and $\SB$ as elements of $\pi_1(U_{B})$, $\pi_1(U_{B\cup B'})$ 
or $\pi_1(U_{B\cup B_F})$ when no ambiguity seems likely to arise.

Let $\gamma_k$ be such that $F_*(\gamma_k)$ is a meridian around a point $P_k\in B_F\setminus (B\cup B')$. Since $F^*(P_k)$ is not a multiple fiber, $\gamma_k$ is trivial in $\pi_1(U_{B})$ by condition~\eqref{dfn:adapted-4} in Definition~\ref{dfn:adapted}. In particular, 
this proves that $\pi_1(U_B)$ can be generated by~$\SF\cup \SBp$. Since $\SB$ is an adapted geometric set of base loops w.r.t $F$ and $B\cup B_F$, then $\SBp$ is an adapted geometric set of base loops w.r.t $F$ and $B\cup B'$.
As in the proof of Lemma~\ref{lemma:admissible}, there exists a meridian in $\pi_1(U_{B \cup B_F})$ about any 
given irreducible component of $C_{P_k}$ of the form $w\gamma_k^{m}$ for some $m\geq 1$, and for $w\in\ker \left(F_{|U_{B \cup B_F}}\right)_*$, which is a word in $\SF$ by Corollary~\ref{cor:Fsurjective-orb}. 
Hence, we have shown that the normal closure of the subgroup generated by the meridians about each irreducible 
component of $C_{P_k}$ is the normal closure of a subgroup of $\pi_1(U_{B \cup B_F})$ generated by $\gamma_k$ and a 
finite number of words in $\SF$. This gives rise to relations of the form~\ref{R3}.

Let $\gamma_k$ be such that $F_*(\gamma_k)$ is a meridian around a point in $B'$.
Similarly as in the previous paragraph, we can choose a meridian in  $\pi_1(U_{B \cup B_F})$ about any of the irreducible components of $C_{P_k}$ which is of the 
form $w(\gamma_k^{m_k})^m$ for some $m\geq 1$, and $w$ a word in $\SF$. Let $N_k$ be the normal closure of the 
subgroup generated by such choice of meridians about each of the irreducible components of $C_{P_k}$. Note that, by 
definition of $\gamma_k\in\pi_1(U_{B\cup B_F})$ and Corollary~\ref{cor:Fsurjective-orb}, the equality $\gamma_k^{m_k}=z_k$ holds in $\pi_1(U_{B})$ for some 
word $z_k$ in $\SF$ (an element of $\ker F_*$).  
Hence, the relations given by $w(\gamma_k^{m_k})^m=1$ coming from the 
chosen generators of $N_k$ (as a normal closure) together with the relation $\gamma_k^{m_k}=z_k$ (which we know holds in 
$\pi_1(U_B)$) are equivalent to a finite number of relations of the form $y=1$ for words $y$ in the letters of $\SF$ (type~\ref{R3}) and 
the relation $\gamma_k^{m_k}=z_k$ (all the relations of type~\ref{R4}).

Let $P_k\in B\setminus B_F$, hence $F_*(\gamma_{k})\in\pi_1(\PP^1\setminus(B\cup B_F))$ induces the trivial monodromy morphism in 
the elements of $\pi_1(F^{-1}(P))$, since $P_k\notin B_F$. In other words, $\gamma_{k}\in \pi_1(U_{B})$ commutes with any word in~$\SF$.

The result now follows from adding the relations found in the previous paragraphs to the presentation of $\pi_1(U_{B\cup B_F})$ 
given in Lemma~\ref{lemma:semidirect}: the relation $\prod_j x_j=\prod_i[a_i,b_i]$ is of type~\ref{R3}, the monodromy relations of $\pi_1(U_{B\cup B_F})$ corresponding to $P_k$ are the ones of 
type~\ref{R2} if $P_k\in (B_F\cap B)\cup B')$, of type~\ref{R1} if $P_k\in B\setminus B_F$ (by the previous paragraph), 
or become of type~\ref{R3} after using that $\gamma_k$ is trivial in $\pi_1(U_B)$ if $P_k\in B_F\setminus(B\cup B')$.
\end{proof}

Our next goal is to describe cases in which $\pi_1(U_B)$ has a presentation on generators $\SBp$. This will provide candidates 
for $B\subset S$ such that $F_*:\pi_1(U_B)\to\pi_1^{\orb}(S_{(n+1,\bar m)})$ is an isomorphism (and, in particular, such that 
Our next goal is to describe cases in which $\pi_1(U_B)$ has a presentation on generators $\SBp$. This will provide candidates
for $B\subset S$ such that $F_*:\pi_1(U_B)\to\pi_1^{\orb}(S_{(n+1,\bar m)})$ is an isomorphism (and, in particular, such that
$\pi_1(U_B)$ is a free product of cyclic groups). This goal is achieved in Proposition~\ref{prop:presentationgamma} with the help of
Lemma~\ref{lemma:KequalsN}. We follow notation from Section~\ref{sec:pi1-fiber-type} and Lemma~\ref{lem:presentationexplicit}.

\begin{lemma}\label{lemma:KequalsN}
Let $F:U\to S$ be an admissible map, $B=\{P_0,\dots,P_n\}\subset S$ a non-empty set, and assume 
$Q\in B\setminus \left(B_F\cap B\right)$. Let $S_{(n+1,\bar m)}$ (resp. $S_{(n,\bar m)}$) be the maximal orbifold 
structure of $S$ w.r.t. $F_|:U_{B}\to S\setminus B$ (resp. $F_|:U_{B\setminus\{Q\}}\to S\setminus (B\setminus\{Q\})$).
Consider $K:=\ker F_*$ the kernel of $F_*:\pi_1(U_{B})\to\pi_1^{\orb}(S_{(n+1,\bar m)})$.

Moreover, assume that $F_*:\pi_1(U_{B\setminus\{Q\}})\to\pi_1^{\orb}(S_{(n,\bar m)})$ is an isomorphism, and furthermore, either
\begin{itemize}
\item $n\geq 1$ or
\item $n=0$ and $B'\neq \emptyset$, where $B':=M_F\setminus B\neq \emptyset$.
\end{itemize}
Then $K$ is an abelian group. Furthermore, an adapted geometric set of base loops $\SBp=\SB(B\cup B')$ w.r.t $F$ and $B\cup B'$ can be chosen so that~$K=N$, where $N$ is the normal closure of the subgroup
$\langle \gamma_k^{m_k}\mid P_k\in B'\rangle$.
\end{lemma}

\begin{proof}
We use the presentation of $\pi_1(U_B)$ given in Lemma~\ref{lem:presentationexplicit} and the notation therein. Since $\gamma_k^{m_k}\in K$ for all $P_k\in B'$ and $K$ is normal, $N\leq K$ for every choice of $\SBp$ as in Lemma~\ref{lem:presentationexplicit}.

Assume now that $F_*$ is an isomorphism and either $n\geq 1$ or $n=0$ and $B'\neq \emptyset$. 
We will show that $K$ is abelian and $K\leq N$ for some choice of $\SBp$. 
Let $\gamma\in\pi_1(U_B)$ be any positively oriented meridian about the 
irreducible fiber $C_Q$, such that $F_*(\gamma)$ is a positively oriented meridian
around $Q$ in $S\setminus (B\cup B_F)$. Note the following:
\begin{enumerate}
\item Since $C_Q$ is irreducible, $\ncl(\gamma)$ is independent of the choice of the 
meridian $\gamma$ by Lemma~\ref{lemma:meridians-conjugated}.
\item \label{stepcommute} 
$K\unlhd \ncl(\gamma)\unlhd \pi_1(U_{B})$ is a subgroup of the normal closure of $\langle \gamma\rangle$
in $\pi_1(U_{B})$. To see this consider $w\in K$. The projection of $w$ in $\pi_1(U_{B\setminus\{Q\}})$ 
is trivial by Corollary~\ref{cor:Fsurjective-orb}, so $w$ is an element of $\ncl(\gamma)$. 
\item \label{stepabelian}
$K\unlhd \pi_1(U_{B})$ is abelian. To see this note that any meridian around the typical fiber $C_Q$ commutes 
with $K=\iota_*(F^{-1}(P))$ by Lemma~\ref{lem:presentationexplicit}\ref{R1}. 
By Lemma~\ref{lemma:meridians-generators} this means $\ncl(\gamma)$ is contained in the centralizer of $K$. 
By~\eqref{stepcommute} $K\leq \ncl(\gamma)$, in particular, $K$ is an abelian subgroup.
\item \label{steprewrite} 
There exists an adapted geometric set of base loops $\SBp$ w.r.t. $F$ and $B\cup B'$ such that one such positively oriented meridian $\gamma$ can be written as a word in $\SBp$.

If $n\geq 1$, we may assume $Q=P_1$ and $\gamma=\gamma_1$ in $\SBp$, which concludes~\eqref{steprewrite} in the case $n\geq 1$.

Suppose that $n=0$ and $b'=\# B' \geq 1$. In this case, $r=2g_S$. Let $\SBp=\SB(B\cup B')$ be as in 
Lemma~\ref{lem:presentationexplicit}. By definition of $\SBp$,  
if $\tilde\gamma=\prod_{i=1}^{g_S}[\gamma_{b'+2i-1},\gamma_{b'+2i}]\cdot\left(\prod_{P_k\in B'}\gamma_k\right)^{-1}$
(see condition~\eqref{dfn:adapted-3} in Definition~\eqref{dfn:adapted}), then $F_*\left(\tilde\gamma\right)$ is a meridian around $Q$ in 
$\pi_1(S\setminus(B\cup B'))$. By Corollary~\ref{cor:Fsurjective-orb} there exists a word $z$ in the letters $\SF$ 
such that $\gamma=\tilde\gamma z$ is a meridian around $C_Q$ whose image by $F_*$ is a meridian around $Q$.
After replacing $\gamma_1$ by $z^{-1}\gamma_1$ in $\SBp$ one can assume $\gamma=\tilde\gamma$.
This concludes~\eqref{steprewrite} in the case $n=0$ and $b'\geq 1$.
\end{enumerate}

Finally, let's show $K\leq N$, where $N$ is defined using the $\SBp$ found in observation~\eqref{steprewrite} above. By~\eqref{steprewrite}, $\gamma=w(\gamma_1,\dots,\gamma_{r+b'})$ is a word in $\SBp$. 
Let $\varphi:\FF_{r+b'}=\langle \delta_1,\ldots,\delta_{r+b'}\rangle\to \pi_1(U_{B})$ 
be the group homomorphism given by $\delta_i\mapsto\gamma_i$ 
for all $i\in\{1,\ldots,r+b'\}$. Let $\alpha$ be an element in $\ncl(\gamma)$ and write $\alpha$ as a product 
$\prod_{i=1}^{u}g_i^{-1}\gamma^{d_i}g_i$, where $g_i$ is an element of $\pi_1(U_{B})$. By Lemma~\ref{lemma:semidirect} 
applied to $\pi_1(U_{B \cup B_F})$, $g_i$ in $\pi_1(U_{B})$ can be written as $g_i=w_ih_i$, where $w_i$ is a word in $\SF$ 
(so $w_i\in K$) and $h_i$ is a word in $\SBp$. Since $\gamma$ commutes with the 
elements of $K\leq\pi_1(U_{B})$, $\alpha$ can be written as the product $\prod_{i=1}^{u}h_i^{-1}\gamma^{d_i}h_i$. 
In other words, we have shown that $\ncl(\gamma)=\varphi(\ncl(a))$, where $\ncl(a)$ is the normal closure in 
$\FF_{r+b'}$ of the subgroup generated by $a$, where $a=w(\delta_1,\dots,\delta_{r+b'})$ (the word $w$ for $\gamma$
in the new letters $\delta_1,\dots,\delta_{r+b'})$. In particular, one has $K\leq\varphi(\ncl(a))$.

A similar argument, this time using that $K$ is abelian, shows that
\begin{equation}\label{eq:Nvarphi}
N=\varphi\left(\ncl\left(\langle \delta_k^{m_k} \mid P_k\in B'\rangle\right)\right).
\end{equation}

Now, the composition $F_*\circ \varphi$ induces an isomorphism in the quotients given by the composition of
$$
\FF_{r}*\left(\mathop{*}_{P_k\in B'}\ZZ_{m_k}\right)\cong
\FF_{r+b'}/\ncl\left(\langle \delta_j^{m_k} \mid P_k\in B'\rangle\right)\to \pi_1(U_{B})/N
$$
with
$$
\pi_1(U_{B})/N\to \pi_1^{\orb}(S_{(n+1,\bar m)})\cong\FF_{r}*\left(\mathop{*}_{P_k\in B'}\ZZ_{m_k}\right).
$$
This implies that $\ker(F_*)\cap\im(\varphi)/N$ is trivial. Recall that $K=\ker(F_*)$ and that, 
$K\leq\varphi\left(\ncl\left(a\right)\right)$. Hence $K=\ker(F_*)\cap\im(\varphi)$ and thus 
we arrive at the equality $K= N$. This concludes the proof.
\end{proof}

\begin{prop}\label{prop:presentationgamma}
Under the same notation as Lemma~\ref{lem:presentationexplicit} and the same hypotheses and choice of $\SBp$ as
Lemma~\ref{lemma:KequalsN}, every element of $K$ can be written as a word in the letters $\SBp$. In particular,
the presentation of $\pi_1(U_B)$ from Lemma~\ref{lem:presentationexplicit} can be transformed to a presentation on
generators~$\SBp$.
\end{prop}
\begin{proof}
Consider the presentation of $\pi_1(U_B)$ given in Lemma~\ref{lem:presentationexplicit}. The generators of 
$\pi_1(U_B)$ in $\SF$ are elements of $K$, which equals $N$ by Lemma~\ref{lemma:KequalsN}. Equation~\eqref{eq:Nvarphi} 
in the proof of Lemma~\ref{lemma:KequalsN} implies that the elements of $N$ are products of elements of the form 
$v\gamma_{k}^{m_k}v^{-1}$, where $v$ is a word in $\SBp$, and $P_k\in B'$. Using this, one can eliminate
the generators of $\pi_1(U_B)$ in $\SF$ in the presentation given in Lemma~\ref{lem:presentationexplicit}.
\end{proof}

\begin{rem}\label{remark:ngeq1}
Suppose that $n\geq 1$, and assume the hypotheses, choice of $\SBp$ and notation of Lemma~\ref{lemma:KequalsN}. In the proof of
Lemma~\ref{lemma:KequalsN}, $\SBp$ was chosen so that $\gamma=\gamma_1\in \SBp$.
Since $\ncl(\gamma)$ is contained in the centralizer of $N=K$ by~\eqref{stepabelian} in the proof of 
Lemma~\ref{lemma:KequalsN}, we can assume that the elements of $N$ are products of elements of the form
$v\gamma_{k}^{m_k}v^{-1},  P_k\in B'$ where $v$ is a word in $\SBp\setminus \{\gamma_1\}$,
i.e. the letter $\gamma_1=\gamma$ does not appear in~$v$.
\end{rem}

\begin{rem}\label{remark:nequals0}
Under the hypotheses, choice of $\SBp$ and notation of Lemma~\ref{lemma:KequalsN}, suppose moreover that $U$ is simply connected, 
$B=\{Q\}$ ($Q\notin B_F$), and $M_F\neq\emptyset$. Note that in this case $r=n=0$ (Remark~\ref{rem:h1-torsion}) 
and $B'=M_F$. In this setting, Corollary~\ref{cor:presentation} implies $\ncl(\gamma)=\pi_1(U_B)$.
By~\eqref{stepabelian} in the proof of Lemma~\ref{lemma:KequalsN}, one has $K=N$ is contained in the center of 
$\pi_1(U_B)$. In particular, any subgroup generated by elements in $K$ is normal, and thus
$$
N=\langle \gamma_k^{m_k}\mid {P_k\in M_F}\rangle.
$$
\end{rem}

The following two corollaries pertain to the case $n\geq 1$ (Corollary~\ref{cor:presentationnleq1}) and 
$n=0, B'=M_F\setminus B\neq \emptyset$ (Corollary~\ref{cor:presentationn=0simply}) in Lemma~\ref{lemma:KequalsN} 
respectively, and give useful presentations of $\pi_1(U_B)$. 
Note that in Corollary~\ref{cor:presentationn=0simply}, $S=\PP^1$ and $U$ are both assumed to be simply connected.

\begin{cor}\label{cor:presentationnleq1}
Assume $F:U\to S$ is an admissible map. Let $B=\{P_0,P_1,\ldots,P_n\}\subset S$ be such that $\# B=n+1\geq 2$. 
Let $S_{(n+1,\bar m)}$ (resp. $S_{(n,\bar m)}$) be the maximal orbifold structure of $S$ with 
respect to $F_|:U_{B}\to S\setminus B$ (resp. $F_|:U_{B\setminus\{P_1\}}\to S\setminus (B\setminus\{P_1\})$).

Suppose that $F_*:\pi_1(U_{B\setminus\{P_1\}})\to\pi_1^{\orb}(S_{(n,\bar m)})$ is an isomorphism, and that $P_1\notin B_F$. 
Let $B'=M_F\setminus B$, and let $\SBp=\SBp(B\cup B')$ be an adapted geometric set of base loops w.r.t. $F$ and $B\cup B'$ as in Remark~\ref{remark:ngeq1}. Then $\pi_1(U_B)$ has a finite presentation
$$
\pi_1(U_B)=\langle\ \SBp : \{R_j\}_{j\in J}, \{\widetilde{R_i}\}_{i\in I} \rangle,$$
where $R_j$ is a word in $\SBp\setminus \{\gamma_1\}$ for all $j\in J$ and $\widetilde{R_i}=[\gamma_{1},w_i]$ 
for all $i\in I$, where $w_i$ is a word in~$\SBp\setminus \{\gamma_1\}$.
\end{cor}
\begin{proof}
Consider the presentation of $\pi_1(U_B)$ explained in the proof of Proposition~\ref{prop:presentationgamma} on
generators $\SBp$, which arises from the presentation in Lemma~\ref{lem:presentationexplicit}. Using 
Remark~\ref{remark:ngeq1}, we see that all of the relations appearing in our presentation are either of type 
$\widetilde{R}_i$ (relations \ref{R1} in Lemma~\ref{lem:presentationexplicit}, for $k=1$) or of type $R_j$ 
(rest of the relations in Lemma~\ref{lem:presentationexplicit}).
\end{proof}

\begin{cor}
\label{cor:presentationn=0simply}
Assume $F:U\to \PP^1$ is an admissible map, where $U$ is a simply connected quasi-projective surface. 
Let $Q\in \PP^1\setminus B_F$, and let $\PP^1_{(1,\bar m)}$ be the maximal orbifold structure of $\PP^1$ 
with respect to $F_|:U_{\{Q\}}\to \PP^1\setminus \{Q\}$. Assume that $M_F\neq\emptyset$. 

Let $\SBp$ be an adapted geometric set of base loops w.r.t. $F$ and let $\{Q=P_0\}\cup M_F$ be given by Lemma~\ref{lemma:KequalsN}.
Then, $\pi_1(U_{\{Q\}})$ has a finite presentation
$$
\pi_1(U_{\{Q\}})=\langle\ \SBp : \{R_j\}_{j\in J}, \{[\gamma_k,\gamma_i^{m_i}]\}_{P_k,P_i\in M_F} \rangle,
$$
where $m_k$ is the multiplicity of the fiber $F^*(P_k)$ and $R_j=\prod_{P_k\in M_F}(\gamma_k^{m_k})^{n_{kj}}$
for some $n_{kj}\in\ZZ$, $j\in J$.
\end{cor}

\begin{rem}\label{rem:A12}
By Remark~\ref{rem:2fibers}, note that $\# M_F=1$ or $2$ in Corollary~\ref{cor:presentationn=0simply}.
\end{rem}

\begin{proof}[Proof of Corollary~\ref{cor:presentationn=0simply}]
Since $U$ is simply connected and $F$ is admissible, $\pi_1^{\orb}(\PP^1_{(0,\bar m)})$ must be trivial by Remark~\ref{rem:inducedorb}.
Hence, $F_*:\pi_1(U)\to\pi_1^{\orb}(\PP^1_{(0,\bar m)})$ is trivially an isomorphism, and the hypotheses of Lemma~\ref{lemma:KequalsN} are satisfied.

By Remark~\ref{remark:nequals0}, the elements of $K=N$ (an abelian group) are all of the form \linebreak
$\prod_{P_k\in M_F}\left(\gamma_k^{m_k}\right)^{n_k}$, where $n_k\in\ZZ$.

Since Remark~\ref{remark:nequals0} says that $N$ is in the center of $\pi_1(U_{\{Q\}})$,
we can add the relations $[\gamma_k,\gamma_i^{m_i}]$ for all $P_i,P_k\in M_F$ to the
presentation of $\pi_1(U_{\{Q\}})$ of Proposition~\ref{prop:presentationgamma} without
changing the group. After that, note that we can transform the
relations already appearing in the presentation of Proposition~\ref{prop:presentationgamma}
(coming from \ref{R2}--\ref{R4}
in Lemma~\ref{lem:presentationexplicit}) to elements of $K=N$, and hence, as relations
of type~$R_j$.
\end{proof}

\subsection{Deletion Lemma}

\begin{thm}[Deletion Lemma]
\label{lemma:removing-fiber-orbi}
Let $U$ be a smooth quasi-projective surface and let  $F:U\rightarrow S$ be an admissible
map to a smooth projective curve $S$.
Assume $B\subset S$ is such that $\# B=n\geq 1$ and $r:=2g_S+n$. Consider $P\in S\setminus B$.
Let $S_{(n+1,\bar m)}$
(resp. $S_{(n,\bar m')}$) be the maximal orbifold structure of $S$ with respect to
$F_|:U_{B\cup\{P\}}\to S\setminus (B\cup\{P\})$
(resp. $F_|:U_{B}\to S\setminus B$).

If 
$F_*:\pi_1(U_{B\cup\{P\}})\to \pi_1^{\orb}(S_{(n+1,\bar m)})$ 
is an isomorphism, then
$$F_*:\pi_1(U_{B})\to \pi_1^{\orb}(S_{(n,\bar m')})$$ is an isomorphism.

Moreover, if
$\pi_1(U_{B\cup\{P\}})\cong \FF_{r}*\ZZ_{m_1}*\dots*\ZZ_{m_s},$
and $p\geq 1$ denotes the multiplicity of $F^*(P)$, then
$$\pi_1(U_{B})\cong 
\FF_{r-1}*\ZZ_{p}*\ZZ_{m_1}*\dots*\ZZ_{m_s}.
$$
\end{thm}

\begin{proof}
By hypothesis,
$\pi_1(U_{B\cup\{P\}})\cong\pi_1(S_{(n+1,\bar m)})\cong\FF_{r}*\ZZ_{m_1}*\dots*\ZZ_{m_s}$ for integers $m_i\geq 2$, 
$i\in I=\{1,\dots,s\}$. Also, by Corollary~\ref{cor:Fsurjective-orb}, the inclusion of 
the generic fiber (over a point $Q\in S\setminus (B\cup\{P\}\cup B_F)$) induces the trivial morphism 
$\pi_1(F^{-1}(Q))\to \pi_1(U_{B\cup\{P\}})$. Since $U_{B\cup\{P\}}\subset U_{B}$, the inclusion of the generic fiber also 
induces the trivial morphism $\pi_1(F^{-1}(Q))\to \pi_1(U_{B})$. By Corollary~\ref{cor:Fsurjective-orb},
$$
F_*:\pi_1(U_{B})\to \pi_1^{\orb}(S_{(n,\bar m')})
$$
is an isomorphism since $n>0$.

For the \emph{moreover} part, note that if $\bar m=(m_1,\ldots, m_s)$, then
\[
\begin{gathered}[b]
\bar m':=
\begin{cases}
\bar m & \textrm{ if } p=1,\\
(p,m_1,\dots,m_s) & \textrm{ if } p>1. 
\end{cases}
\vspace*{-5pt}
\end{gathered}
\qedhere
\]
\end{proof}

\begin{rem}
The Deletion Lemma also holds in the case $n=0$ in the following way: if $F_*:\pi_1(U_{\{P\}})\to \pi_1^{\orb}(S_{(1,\bar m)})$ is an isomorphism, then $F_*:\pi_1(U)\to \pi_1^{\orb}(S_{(0,\bar m')})$ is also an isomorphism. A more subtle proof can be given using the presentation of $U_{\{P\}}$ from 
Lemma~\ref{lem:presentationexplicit}, taking into account that the elements of $\Gamma_F$ are trivial in $\pi_1(U_{\{P\}})$. 
Since the result is not needed for the purpose of this paper, we omit it.
\end{rem}

\subsection{Proof of the Generic Addition-Deletion Lemma}

\begin{proof}[Proof of Theorem~\ref{lemma:generic-fiber-orbi}]
The ``if'' as well as the ``moreover'' parts of the statement are a particular case of the 
Deletion Lemma~\ref{lemma:removing-fiber-orbi}.

Let us show the ``only if'' part. Let $P_1=P$, $B\cup\{P\}=\{P_0,\ldots,P_{n}\}$. By Corollary~\ref{cor:presentationnleq1} 
(and using the notation therein) applied to $B$ and $B\cup\{P\}$, we have that $\pi_1(U_{B\cup\{P\}})$ has a presentation of the form
$$
\pi_1(U_{B\cup\{P\}})=\langle\ \SBp: \{R_j\}_{ j\in J}, \{\widetilde{R}_i\}_{i\in I}\rangle
$$
where $R_j$ are words in $\SBp\setminus \{\gamma_1\}$, and $\widetilde{R}_i=[\gamma_{1},w_i]$, 
where $w_i$ is a word in $\SBp\setminus \{\gamma_1\}$.

Let $H=\langle\ \SBp\setminus \{\gamma_1\} : \{R_j\}_{j\in J}\rangle$, and let 
$\varphi:\FF_{r+b'}=\langle\delta_1,\ldots,\delta_{r+b'}\rangle\to\pi_1(U_B)$ be the epimorphism sending $\delta_i$ to 
$\gamma_i$ for all $i=1,\ldots,r+b'$. We have that $\varphi$ factors through $\widetilde{\varphi}:\ZZ*H\to \pi_1(U_{B\cup\{P\}})$, 
where the $\ZZ$ free factor is generated by the letter $\gamma_{1}$. In particular $\widetilde{\varphi}$ is an epimorphism.

Moreover, according to the presentation of $\pi_1(U_{B\cup\{P\}})$ above,
$$
\pi_1(U_{B})\cong \pi_1(U_{B\cup\{P\}})/\ncl(\gamma_{1})\cong H
$$
Hence, we have found an epimorphism
$$
\widetilde{\varphi}:\ZZ*\pi_1(U_B)\to \pi_1(U_{B\cup\{P\}}).
$$
Let $F_*:\pi_1(U_{B\cup\{P\}})\to \pi_1^{\orb}(S_{(n+1,\bar m)})$. Since $\widetilde{\varphi}$ is an epimorphism and $F_*$ is 
also an epimorphism by Remark~\ref{rem:inducedorb}, $F_*\circ\widetilde{\varphi}$ is an epimorphism from $\ZZ*\pi_1(U_B)$ to $\pi_1^{\orb}(S_{(n+1,\bar m)})$, which is a free product of cyclic groups  isomorphic to
$\ZZ*\pi_1^{\orb}(S_{(n,\bar m)})\cong\ZZ*\pi_1(U_B)$. Hence, $F_*\circ\widetilde{\varphi}$ is an isomorphism by 
Lemma~\ref{lemma:EpiImpliesIso}. In particular, $\widetilde{\varphi}$ and $F_*$ are both isomorphisms, and 
\[\pi_1(U_{B\cup\{P\}})\cong\ZZ*\pi_1(U_B).\qedhere\]
\end{proof}

\begin{cor}\label{cor:polynomialmap}
The fundamental group of the complement of $r$ generic fibers of a primitive polynomial map $f:\CC^2\to\CC$ is 
free of order~$r$.
\end{cor}

\begin{proof}
Consider $F:\PP^2\dashrightarrow\PP^1$, $F(x,y,z)=[\overline f(x,y,z):z^d]$, where $\overline f(x,y,z)$ is the 
homogenization of $f(x,y)$ and $d=\deg(f)$. Since $f$ is primitive (it has connected generic fibers), $F$ is an admissible map.
Let $\cB$ be the base points of $F$. Consider $U=\PP^2\setminus \cB$, and $F_|:U\to\PP^1$.
The restriction $F_|:\CC^2\to\CC$ induces an isomorphism of (trivial) fundamental groups. 
Note that $F_|:\CC^2\to\CC$ does not have multiple fibers, or else $F_*$ would factor through a surjection from
$\pi_1(\CC^2)$ to $\pi_1^{\orb}(\CC)$ for an orbifold on $\CC$, which is non-trivial. Consider $r$ generic fibers 
of $F_|$. Then the result follows from the Generic Addition-Deletion Lemma~\ref{lemma:generic-fiber-orbi}
to $F_|:U\to\PP^1$ and $B=\{[1:0]\}$.
\end{proof}

\begin{exam}
\label{exam:family}
Other examples of complements of curves with free fundamental groups include the following. 
Consider a polynomial $f(x,y)$ such that $\pi_1(\CC^2\setminus C_1)=\ZZ$ for $C_1=V(f=0)\subset \CC^2$ (for instance, if $C_1$ is irreducible and only has nodal singularities, including at infinity).
We have that $f_*:\pi_1(\CC^2\setminus C_1)\to\pi_1(\CC^*)$ is an epimorphism from $\ZZ$ to itself, so it is an isomorphism. 
Note that $f$ does not have multiple fibers, or else $f_*$ would factor through a surjection from
$\pi_1(\CC^2\setminus C_1)\cong\ZZ$ to $\pi_1^{\orb}(\CC^*)$ for an orbifold of general type on $\CC^*$, which is a 
non-abelian group. Consider $C_2,\dots,C_r$ generic 
fibers of $f$. Then the Generic Addition-Deletion Lemma~\ref{lemma:generic-fiber-orbi} yields $\pi_1(\CC^2\setminus C)=\FF_r$ 
for $C=C_1\cup\dots\cup C_r$.

Analogously, if $f_p$ (resp. $f_q$) is a form of degree $p$ (resp. $q$) with $\gcd(p,q)=1$ and
$\pi_1(\PP^2\setminus C_1)=\ZZ$ for $C_1=V(f_p)\cup V(f_q)\subset\PP^2$. This implies that $f_p$ and $f_q$ are irreducible
and in that case $F=[f_p^q:f_q^p]$ is also an admissible map (see for instance~\cite[Lemma 2.6]{ji-Eva-Zariski}).
Consider $C_2,\dots,C_r$ generic
fibers of $F=[f_p^q:f_q^p]$. Then we are under the hypotheses of the Generic Addition-Deletion Lemma~\ref{lemma:generic-fiber-orbi}
and hence $\pi_1(\PP^2\setminus C)=\FF_r$ for $C=C_1\cup\dots\cup C_r$.
\end{exam}

\subsection{A base case for the Addition Lemma}

Recall Notation~\ref{not:XB}. In light of Example~\ref{exam:family}, one might wonder if other pencils 
$F:X=\PP^2\dashrightarrow \PP^1$ give rise to curves whose fundamental group of their complement is isomorphic 
to an open orbifold group of $\PP^1$ through the morphism induced by $F$. The following result provides a base case for the Addition Lemma in the case $M_F\neq \emptyset$. 

\begin{thm}\label{thm:basecasesimply}
Let $X$ be a simply connected smooth projective surface, let  $F:X\dasharrow \PP^1$ be an admissible map, and let $P\in\PP^1$ 
be such that $F^{-1}(P)$ is a typical fiber. Suppose that $M_F\neq \emptyset$. 

Let $\PP^1_{(1,\bar m)}$ be the maximal orbifold structure of $\PP^1\setminus \{P\}$ with respect to 
$F:X_{\{P\}}\to\PP^1\setminus \{P\}$, where $\bar m=(p,q)$, $p\geq q\geq 1$, and $\gcd(p,q)=1$. Then the following 
statements are equivalent:
\begin{enumerate}
 \item \label{item:1basecasesimply}$H_1(X_{\{P\}})=\ZZ_{pq}$
 \item \label{item:2basecasesimply}$F_*:\pi_1(X_{\{P\}})\to\pi_1^{\orb}(\PP^1_{(1,\bar m)})\cong\ZZ_p*\ZZ_q$ is an isomorphism.
\end{enumerate}
\end{thm}

\begin{proof}
After an isomorphism in $\PP^1$, we may assume that $M_F\subset\{[0:1],[1:0]\}$ (see Remark~\ref{rem:2fibers}), 
the fiber above $[0:1]$ has multiplicity $p$, the fiber above $[1:0]$ has multiplicity $q$, and $P\notin\{[0:1],[1:0]\}$.

\eqref{item:2basecasesimply}$\Rightarrow$ \eqref{item:1basecasesimply} is trivial. Let us prove 
\eqref{item:1basecasesimply}$\Rightarrow$ \eqref{item:2basecasesimply}.

Assume that $H_1(X_{\{P\}})=\ZZ_{pq}$. Note that
\begin{itemize}
\item $\#M_F=2$ if and only if $q>1$, and
\item $\#M_F=1$ if and only if $q=1$.
\end{itemize}
Corollary~\ref{cor:presentationn=0simply} applied to $U=\PP^2\setminus\mathcal B$ and Remark~\ref{rem:A12} say that 
$\pi_1(X_{\{P\}})$ has a presentation of the form
\begin{equation}\label{eq:presn=0}
\pi_1(X_{\{P\}})=\langle \gamma_1,\gamma_2\mid \{\gamma_1^{pk_j}\gamma_2^{ql_j}\}_{j\in J}, 
[\gamma_1,\gamma_2^q], [\gamma_2,\gamma_1^{p}]\rangle
\end{equation}
where $k_j,l_j\in\ZZ$ for all $j\in J$. Indeed, this is clear if $\#M_F=b'=2$ (i.e. $q>1$), but it is also true 
if $b'=1$ (i.e. $q=1$), picking $k_1=0$, $l_1=1$. Hence, from now on we assume that $\pi_1(X_{\{P\}})$ has a 
finite presentation as in equation~\eqref{eq:presn=0}.

Let $k=\gcd_{j\in J}(k_{j})$, and $l=\gcd_{j\in J}(l_j)$, with the convention that the greatest common divisor of 
various $0$'s is $0$. Note that this group has a quotient
$$
\langle \gamma_1,\gamma_2 \mid \gamma_1^{pk}, \gamma_2^{ql}, [\gamma_1,\gamma_2]\rangle,
$$
so the quotient map induces an epimorphism on the abelianizations
$$
\ZZ_{pq}\twoheadrightarrow \ZZ_{pk}\times\ZZ_{ql}.
$$
Hence, $k=l=1$.
Using that $\gamma_1^{p}$ and $\gamma_2^{q}$ commute, we can modify the presentation of $\pi_1(X_{\{P\}})$ in 
equation~\eqref{eq:presn=0} to get
\begin{equation}
\label{eq:presn=02}
\langle \gamma_1,\gamma_2 \mid \gamma_1^p\gamma_2^{aq},\gamma_2^{bq}, [\gamma_1,\gamma_2^q], [\gamma_2,\gamma_1^{p}]\rangle,
\end{equation}
where $b\geq 1$, $a\in\{0,\ldots,b-1\}$ and $\gcd(a,b)=1$. Thus, a presentation matrix of the abelianization of 
$\pi_1(X_{\{P\}})$ as a $\ZZ$-module is given by
$$
M=\left(
\begin{array}{cc}
p & 0\\
qa & qb
\end{array}\right)
$$
By hypothesis, we know that the matrix $M$ is equivalent (over $\ZZ$) to the diagonal $2\times 2$ matrix with diagonal $(p,q)$, 
since both matrices present the same abelian group and have the same dimensions and rank. In particular, both matrices have the 
same determinant, so $b=1$, and thus $a=0$. Plugging that data back in for the presentation in equation~\eqref{eq:presn=02}, 
we get that
$$
\pi_1(X_{\{P\}})=\langle \gamma_1, \gamma_{2} \mid \gamma_1^{p}, \gamma_2^{q}\rangle,
$$
and the epimorphism (recall Remark~\ref{rem:inducedorb})
$$
F_*:\pi_1(X_{\{P\}})\to \pi_1^{\orb}(S_{(1,\bar m)})
$$
is in fact an isomorphism by Lemma~\ref{lemma:EpiImpliesIso}.
\end{proof}

\section{Applications}\label{sec:applications}

\subsection{$C_{p,q}$-curves revisited}\label{sec:Cpq}
In this subsection we prove a generalization of Oka's classical result on $C_{p,q}$ curves.

\begin{proof}[Proof of Theorem~\ref{thm:Okapq}]
If $p=q=1$, the result is trivial. Suppose that $p\geq 1$
or $q>1$, which implies that $M_F\neq \emptyset$. Let $P\in\PP^1$ be such that $F^{-1}(P)$ is a typical fiber. In particular, 
$\overline{F^{-1}(P)}$ is given by the zeros of an irreducible degree $pq$ polynomial, and by Remark~\ref{rem:coprime}, 
$H_1(X_{\{P\}})\cong\ZZ_{pq}$. In other words, \eqref{item:2basecasesimply} in Theorem~\ref{thm:basecasesimply} holds, 
and the result for a finite union of generic fibers is proved for $r=0$. The claim for $r>0$ follows from the 
Generic Addition-Deletion Lemma~\ref{lemma:generic-fiber-orbi} applied to $U=\PP^2\setminus \mathcal{B}$. 

Let $\PP^1_{\bar m}$ be the maximal orbifold structure of $\PP^1\setminus\{[0:1]\}$ with respect to 
$F_|:\PP^2\setminus V(f_p)=U_{\{[0:1]\}}\to\PP^1\setminus\{[0:1]\}$, and note that $\pi_1^{\orb}(\PP^1_{\bar m})\cong \ZZ_p$.
Assume that $\pi_1(\PP^2\setminus V(f_p))\cong \ZZ_p$. In that case, 
$F_*:\pi_1(\PP^2\setminus V(f_p))\to \pi_1^{\orb}(\PP^1_{\bar m})$ is an epimorphism between $\ZZ_p$ and itself, 
so it is an isomorphism. By the Generic Addition-Deletion Lemma~\ref{lemma:generic-fiber-orbi} one obtains
$\pi_1\left(\PP^2\setminus (C\cup V(f_p))\right)=\FF_{r+1}*\ZZ_p$.
\end{proof}

\begin{rem}
Examples of these families also appear in Exercise \S4(4.21) of Dimca's reference book~\cite{Dimca-singularities}. 
Theorem~\ref{thm:Okapq} can serve as a proof for part (ii) of this exercise.
\end{rem}

\subsection{Fundamental group of a union of conics}\label{sec:conics}
Another instance where our results apply is given in a collection of conics in a pencil. We provide a new proof 
of Theorems 2.2 and 2.5 in~\cite{Amram-ErratumFundamentalQuadric} which does not depend on braid monodromy calculations. 

\begin{thm}
\label{thm:Amram}
Let $F=[f_2:f_1^2]$ be a pencil generated by a smooth conic $C_0=V(f_2)$ and a double line $\ell=V(f_1)$.
Consider $C=C_0\cup\dots\cup C_r$ a union of $r+1$ smooth conics of $F$, then 
$$\pi_1(\PP^2\setminus C)=\FF_r*\ZZ_2.$$
\end{thm}

\begin{proof}
This result can be obtained from Theorem~\ref{thm:Okapq} for $p=2$, $q=1$.
\end{proof}

\subsection{Fundamental group of tame torus-type sextics}\label{sec:TorusType}
In a remarkable paper, Oka-Pho~\cite{Oka-Pho-fundamental} describe the fundamental group of the complement of
irreducible sextics in a pencil of type $F=[f_2^3:f_3^2]$, where $f_i$ is a homogeneous form of degree $i$, 
whose set of singular points are base points of the pencil, that is, $\Sing V(f)=V(f_2)\cap V(f_3)$. 
The term \emph{torus type} refers to the former property and the term \emph{tame} refers to the latter.
According to the authors, any such a sextic $C$ satisfies $\pi_1(\PP^2\setminus C)=\ZZ_2*\ZZ_3$ except for 
the particular case where $C$ has four singular points: one of type $C_{3,9}$ and three of type $A_2$. 
Type $C_{3,9}$ singularities have a local equation $f(x,y)=x^3+y^9+x^2y^2\in \CC\{x,y\}$, and $A_2$ singularities
are ordinary cusps of local equation $f(x,y)=x^2+y^3\in \CC\{x,y\}$.

It is enough to check the result on maximal irreducible tame torus-type sextics, that is, those with maximal 
total Milnor number (either 19 or 20).
According to~\cite{Oka-Pho-fundamental} there are seven types of such irreducible curves, which can be described 
by the configuration $\Sigma$ of singularities (see Table~\ref{table:maximal}), since their moduli space is 
connected for each configuration. Also, by the maximality of the total Milnor number, the multiple fibers 
$C_2:=V(f_2)$ and $C_3:=V(f_3)$ are uniquely determined by the moduli space of the curve $V(f_2)\cup V(f_3)$, 
that is, they depend only on the singularities of $C_2$ and $C_3$ and the topological type of their intersection.
Such moduli spaces are connected in all cases.
{\small
\begin{center}
\begin{table}[ht]
\begin{tabular}{|c|c|c|c|}\hline
\rule{0pt}{4ex} 
 & $\Sigma$ & $(\mu,r,\delta)$ & $f_2$, $f_3$, $C=\{f_2^3+f_3^2=0\}$
\\[2ex]\hline\hline
\rule{0pt}{4ex} 
1 & $[C_{3,15}]$ & $[(19,2,10)]$ & $\array{l}f_2=yz-x^2 \\ f_3=40y^3+21xyz-21x^3\endarray$ \\[2ex]\hline
\rule{0pt}{4ex} 
2 & $[C_{9,9}]$ & $[(19,2,10)]$ & $\array{l}f_2=y^2-x^2 \\ f_3=2y^2z-2x^2z+\frac{32}{27}x^3\endarray$ 
\\[2ex]\hline
\rule{0pt}{4ex} 
3 & $[C_{3,7},A_8]$ & $[(11,2,6),(8,1,4)]$ & $\array{l}f_2=yz-x^2 \\
f_3=\frac{23}{27}x^3-\frac{4}{9}x^2y+xyz+\frac{4}{9}xy^2-\frac{4}{27}y^3\endarray$
\\[2ex]\hline
\rule{0pt}{4ex} 
4 & $[Sp_1,A_2]$ & $[(18,1,9),(2,1,1)]$ & $\array{l}f_2=xy \\ f_3=y^2z-y^3-x^3\endarray$
\\[2ex]\hline
\rule{0pt}{4ex} 
5 & $[B_{3,10},A_2]$ & $[(18,1,9),(2,1,1)]$ & $\array{l}f_2=yz-y^2-x^2 \\ 
f_3=y(yz-y^2-x^2+xy+\frac{18}{25}y^2)\endarray$
\\[2ex]\hline
\rule{0pt}{4ex} 
6 & $[B_{3,8},E_6]$ & $[(14,1,7),(6,1,3)]$ & $\array{l}f_2=y^2-2yz+z^2+x^2-z^2 \\ f_3=x^2y\endarray$
\\[2ex]\hline
\rule{0pt}{4ex} 
7 & $[C_{3,9},3A_2]$ & $[(13,2,7),(2,1,1)]$ & $\array{l}f_2=yz-x^2 \\
f_3=y(yz+\frac{4}{3} y^2+\frac{3\sqrt{3}}{2} x+\frac{2\sqrt{3}}{3}xy+y^2)\endarray$
\\[2ex]\hline
\end{tabular}
\vspace*{7pt}
\caption{Configuration of singularities for maximal tame sextics of torus type $(2,3)$.}
\label{table:maximal}
\end{table}
\end{center}}

The following recovers the well known result by Oka-Pho~\cite{Oka-Pho-fundamental} on irreducible maximal 
tame torus sextic of type $(2,3)$ for families 1-6 in Table~\ref{table:maximal}.
For the sake of brevity we will only show the details for family 1, but the same strategy can be followed 
to prove the remaining cases. 

\begin{thm}[\cite{Oka-Pho-fundamental}]
\label{thm:OkaPho}
Let $C=\{f_2^3+f_3^2=0\}$ be an irreducible maximal tame torus sextic of type $(2,3)$ whose configuration of 
singularities $\Sigma_C\neq \{[C_{3,9},3A_2]\}$ (see Table~\ref{table:maximal}). Then
$$
\pi_1(\PP^2\setminus C)=\ZZ_2*\ZZ_3.
$$
\end{thm}

\begin{proof}
The idea of the proof is to show that any irreducible maximal tame torus sextic of type $(2,3)$ except for 
exceptional case $\Sigma_C\neq \{[C_{3,9},3A_2]\}$ (family~$(7)$ in Table~\ref{table:maximal}) is a generic member 
of a primitive pencil satisfying the conditions in Theorem~\ref{thm:Okapq}.

We will do it in detail for $C\in \cM([C_{3,15}])$. Table~\ref{table:maximal} gives a possible equation 
$C=\{f_2^3+f_3^2=0\}$ for such a curve as a member of a pencil generated by a smooth conic $C_2=\{f_2=0\}$ 
and a nodal cubic $C_3=\{f_3=0\}$ whose node $P\in C_2$ is such that $(C_2,C_3)_P=6$, that is, $C_2\cap C_3=\{P\}$
(see~\cite[Thm. 1]{Oka-Pho-classification}). To see that $C$ is in fact a generic member one can obtain the 
resolution of indeterminacies is shown in Figure~\ref{fig:resolution} where
$\hat F^*([0:1])=3C_2+E_{2,1}+2E_{2,2}+3E_{2,3}+E_{2,4}$, $\hat F^*([1:0])=2C_3+E_{3,1}+E_{3,2}$.

\begin{figure}[ht]
\begin{center}
{\small
\begin{tikzpicture}[scale=.6]
\draw (-4,0)--(4,0) node[right] {$D_1(-2)$};
\draw (-4,-6)--(4,-6) node[right] {$D_2(-2)$};
\draw[line width=1.2,color=red] (-2.5,-2)--(-1.5,1) node[right,color=black] {$E_{2,1}(-2)$};
\draw[line width=1.2,color=red] (-2.5,-1)--(-1.5,-3) node[right,color=black] {$E_{2,2}(-2)$};
\draw[line width=1.2,color=red] (-1.5,-2)--(-2.5,-5) node[left,color=black] {$E_{2,3}(-2)$};
\draw[line width=1.2,color=red,arrows=<-] (-2.5,-3) node[above left,color=black] {$C_2(-1)$} to[out=0,in=180] ((-1,-5);
\draw[line width=1.2,color=red] (-2.5,-4)--(-1.5,-7) node[right,color=black] {$E_{2,4}(-3)$};
\draw[line width=1.2,color=blue] (2.5,-2)--(1.5,1) node[right,color=black] {$E_{3,1}(-2)$};
\draw[line width=1.2,color=blue] (2.5,-4)--(1.5,-7) node[right,color=black] {$E_{3,2}(-2)$};
\draw[line width=1.2,color=blue,arrows=<-] (3,-1) node[right,color=black] {$C_3(-1)$} to[out=180,in=180] ((3,-5);
\end{tikzpicture}}
\caption{Resolution of indeterminacy}
\label{fig:resolution}
\end{center}
\end{figure}
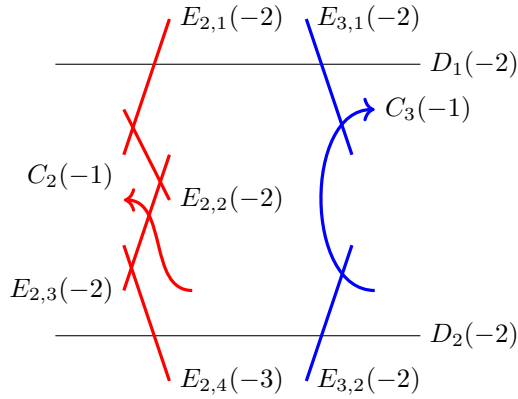

The dicritical divisors $D_1$ and $D_2$ define 1:1 morphisms $\hat F_|:D_i\to \PP^1$ so there is no degeneration of
fibers on the dicritical divisors. One can check that $C$ has no singularities 
outside the base point $P$. This implies that $C$ is a generic fiber.
\end{proof}

\begin{rem}\label{rem:notgen}
Let $C=\{f_2^3+f_3^2=0\}$ be the curve in the moduli space $\cM([C_{3,9},3A_2])$ given by the equation corresponding to 
family~$(7)$ in Table~\ref{table:maximal}. The curve $C$ is not a generic sextic in the pencil $[f_2^3:f_3^2]$. If it were,
Theorem~\ref{thm:basecasesimply} would contradict $\pi_1(\PP^2\setminus C)\ncong\ZZ_2*\ZZ_3$ (Oka-Pho). 
Nonetheless, one can check directly that $C$ is not a typical fiber, since this pencil is of type $(1,6)$, but $C$ is a 
rational curve. This is a consequence of $C$ being irreducible and $\delta(C_{3,9})+3\delta(A_2)=7+3=10$ 
(see Table~\ref{table:maximal}).
\end{rem}

\begin{thm}
\label{thm:OkaPho-1multiple}
Let $F=[f_2^3:f_3^2]:\PP^2\dasharrow \PP^1$ be a pencil such that $C=\{f_2^3+f_3^2=0\}$ is an
irreducible maximal tame torus sextic of type $(2,3)$ and a generic member of the pencil $F$.
Let $\Sigma_C$ be its configuration of singularities. Consider $B\subset\PP^1$ be a collection
of $r+1$ typical values, $C_B=\cup_{\lambda\in B}C_\lambda$, and $C_j=V(f_j)$ for $j=2,3$.
\begin{enumerate}
\item \label{item:OkaPho1} 
If $\Sigma_C\neq \{B_{3,10},A_2\},\{B_{3,8},E_6\},\{C_{3,9},3A_2\}$, namely, if $C$ is a curve in a family $(1)-(4)$, then
$$
\pi_1(\PP^2\setminus C_B\cup C_3)=\FF_{r+1}*\ZZ_3.
$$
\item \label{item:OkaPho2}
If $\Sigma_C\neq \{Sp_1,A_2\}, \{C_{3,9},3A_2\}$, namely, $C$ is a curve in a family $(1)-(3), (5)$ or $(6)$, then
$$
\pi_1(\PP^2\setminus C_B\cup C_2)=\FF_{r+1}*\ZZ_2.
$$
\end{enumerate}
\end{thm}

\begin{proof}
For the proof of part~\eqref{item:OkaPho1}, note that these are the only families where $C_3$ is irreducible.
As mentioned before Table~\ref{table:maximal}, the curves $C_3$ are well defined.
In families $(1)-(3)$, the curve $C_3$ is a nodal cubic, and in family $(4)$, it is a cuspidal cubic transversal to the
line at infinity. In both of these cases, $\pi_1(\PP^2\setminus C_3)\cong \ZZ_3$. 
The result now follows from the Generic Addition-Deletion Lemma~\ref{lemma:generic-fiber-orbi}.

For the proof of part~\eqref{item:OkaPho2}, note that these are the families where $C_2$ is irreducible (a smooth conic). Note that family~$(7)$ also has irreducible $C_2$, but it still does not satisfy the 
hypothesis, since $C$ is not a typical fiber in that pencil (Remark~\ref{rem:notgen}). Since $C_2$ is smooth, 
$\pi_1(\PP^2\setminus C_2)\cong \ZZ_2$. The result follows from the Generic Addition-Deletion 
Lemma~\ref{lemma:generic-fiber-orbi}.
\end{proof}

\end{document}